\documentclass[preprint,12pt,3p]{elsarticle}
\usepackage{array}
\usepackage{amssymb}
\usepackage{amsthm,amsmath}
\usepackage{multirow}
\newtheorem{Lemma}{Lemma}
\newtheorem{theorem}{Theorem}
\usepackage{lineno}
\usepackage{color}
\usepackage{amsmath,dsfont}
\usepackage{graphicx}
\usepackage{float} 
\usepackage{amsmath,dsfont}  
\usepackage[active]{srcltx}
\usepackage[linkcolor=blue, urlcolor=blue, citecolor=blue,colorlinks, bookmarks]{hyperref}
\usepackage{amsfonts}
\usepackage{mathrsfs}
\usepackage[T1]{fontenc}
\usepackage[utf8]{inputenc}
\usepackage{tabularx,ragged2e,booktabs,caption}

\numberwithin{Corollary}{section}
\usepackage{subcaption}
\usepackage{titlesec}
\titleformat*{\subsection}{\bfseries}
\titleformat*{\subsubsection}{\bfseries}

\journal{Proceedings of the National Academy of Sciences, India Section A: Physical Sciences}
\begin{document}
\begin{frontmatter}

\title{Numerical and approximate solutions for two-dimensional hyperbolic telegraph equation via wavelet matrices}
		\author[IITK]{Vijay Kumar Patel\corref{cor1}}
		\ead{vijaybhuiit@gmail.com}
		\author[IITK]{Dhirendra Bahuguna}
		\ead{dhiren@iitk.ac.in}
		\cortext[cor1]{Corresponding author}
		\address[IITK]{Department of Mathematics and Statistics, Indian Institute of Technology kanpur, India}
\begin{abstract}
The present article is devoted to developing the Legendre wavelet operational matrix method (LWOMM) to find the numerical solution of two-dimensional hyperbolic telegraph equations (HTE) with appropriate initial time boundary space conditions. The Legendre wavelets series with unknown coefficients have been used for approximating the solution in both of the spatial and temporal variables. The basic idea for discretizing two-dimensional HTE is based on differentiation and integration of operational matrices. By implementing LWOMM on HTE, HTE is transformed into algebraic generalized Sylvester equation. Numerical experiments are provided to illustrate the accuracy and efficiency of the presented numerical scheme. Comparisons of numerical results associated with the proposed method with some of the existing numerical methods confirm that the method is easy, accurate and fast experimentally. Moreover, we have investigated the convergence analysis of multidimensional Legendre wavelet approximation. Finally we have compared our result with research article of Mittal and Bhatia (see \cite{mittal_2014}). 
\end{abstract} 
\begin{keyword}
			Telegraph equation \sep Legendre wavelets \sep Operational matrices \sep Kronecker multiplications \sep BICGSTAB method.
\end{keyword}
\end{frontmatter}
\section*{Applications}
The study of the electric signal in dispersive wave propagation, a transmission line, pulsating blood flow in arteries and random motion of bugs along a hedge are amongst a host of physical and biological phenomena which can be described by the telegraph partial differential equation (TPDE). So, the TPDE and its solutions are very important in many areas of applications.

\section{Introduction} \label{s2}
Partial differential equations (PDEs) are of widespread interest because of their connection with phenomena in the physical world. Nowadays most of the physical problems can be described in the form of mathematical models and these models are consist of PDEs. PDEs are observed in many fields of engineering and applied sciences. Among these PDEs, hyperbolic PDEs play an important role in several areas of engineering and  applied sciences. The propagation of signal (digital and analog) through media, the propagation of electromagnetic waves in the earth–ionosphere waveguide \cite{kirillov2002solving}, mechanical wave \cite{dehghan2008combination}, an ecological and cosmological phenomena are modeled using hyperbolic PDEs \cite{wollkind1986applications}. Recently, many methodologies have been investigated to find the numerical solution of the telegraph equation due to their universal applications in the area of applied mathematics. Goldstien was the first who derived the one-dimensional telegraph equation with probabilistic argument. He proved that a particle which moves forward  and backward direction with speed c satisfied hyperbolic one-dimensional telegraph equation (for instant see \cite{goldstein1951diffusion}). Hyperbolic telegraph equations are commonly used in wave phenomena \cite{weston_1993} and also wave propagation of electric signal in \cite{jordon_1999},  . And also, an effort has been taken for the extended result for two-dimensional case of random planar motion (see for instant \cite{tolubinsky1969theory,bartlett1978note,orsingher1985hyperbolic}). Specially in two dimensional hyperbolic PDEs such as telegraph equations in real-world applications, we should impose some boundary limitations on the two-dimensional space variable. Neglecting spatial dimension in multidimensional can affect the accuracy of the model for describing the chemical and physical events. Therefore, multidimensional PDEs are considered by scientist and engineer to model and simulate the aforementioned events. Among the PDEs model, parabolic PDEs model can describe a phenomena with the help of some physical laws, but in some of the model, it may be better modeled by hyperbolic PDEs. However, few articles are devoted to the implementation of the analytical methods for solving the telegraph equation with boundary conditions (see \cite{salkuyeh_2011} ). Additionally in different circumference, direct symbolic differentiation and integration in analytical schemes make them time-consuming. This is one more disadvantage of the analytical method for solving multidimensional PDEs. Because of these type of disadvantage in PDEs, the robust and efficient tool should have a look to compute the numerical solution of the problems.

In this article, we consider the more general form of two dimensional hyperbolic telegraph equation as follows
\begin{equation}\label{1}
\begin{split}
\frac{\partial^{2} \Phi(\eta,\xi,t)}{\partial 
t^{2}}+2\lambda_{1}\frac{\partial \Phi(\eta,\xi,t)}{\partial
t}+\lambda_{2}^{2} \Phi(\eta,\xi,t)=\frac{\partial^{2} \Phi(\eta,\xi,t)}{\partial
\eta^{2}}+\frac{\partial^{2} \Phi(\eta,\xi,t)}{\partial
\xi^{2}}\\+F(\eta,\xi,t),~~~(\eta,\xi,t)\in \Omega \times (0,T],
\end{split}
\end{equation}
  where $\Omega=[0,1]\times[0,1]\times[0,1] $ and $t \in (0,T] $.\par
The initial conditions  and the Dirichlet boundary conditions are 
\begin{equation}\label{2}
\left\{ \begin{array}{l} \Phi(\eta,\xi,0)=f_{1}(\eta,\xi), \vspace{0.2cm}\\
\Phi_{t}(\eta,\xi,0)=f_{2}(\eta,\xi), \vspace{0.2cm} 

\end{array}\right.
\end{equation}
and 
\begin{equation}\label{3}
\left\{ \begin{array}{l} \Phi(0,\xi,t)=f_{3}(\xi,t),~\Phi(1,\xi,t)=f_{4}(\xi,t), \vspace{0.2cm}\\
\Phi(\eta,0,t)=f_{5}(\eta,t),~\Phi(\eta,1,t)=f_{6}(\eta,t).~~~~~(\eta,\xi,t)\in
\lambda_{1}\times (0,T] \vspace{0.2cm}
\end{array}\right.
\end{equation}
respectively, where $\lambda_{1}$ and $\lambda_{2}$ are constants. \par

Nowadays, researchers are focused on the numerical method for solving this type of PDEs. Various numerical methods have been developed to solve hyperbolic PDEs and the above telegraph equations (\ref{1})-(\ref{3}) have been considered by some researchers for numerical solutions. Dehghan and Ghesmati have used two meshless methods namely meshless local Petrov-Galerkin (MLPG) and meshless local weak-strong (MLWS) to solve two-dimensional telegraph equations (\ref{1})-(\ref{3}) in \cite{dehghan_2010}. In \cite{gao2007unconditionally}, a semi discretization method which is unconditionally stable has been proposed by F. Gao and C. Chi . R.K. Mohanty developed three-level unconditionally stable schemes based on the finite difference (\cite{mohanty2004unconditionally,mohanty2005unconditionally}). In  \cite{dehghan2011use}, Yousefi and Dehghan used He's variational iteration approach to solve one-dimensional telegraph equation numerically. However, In \cite{mohebbi2008high}, authors used compact finite difference approximation for space derivative. In \cite{lakestani2010numerical}, M. Lakestani and B. N. Saray developed operational matrix approach based on interpolating scaling functions to solve the telegraph equation. Shokri and Dehghan  explored radial basis function in \cite{dehghan2008numerical}, a meshless method in \cite{dehghan2009meshless} and a meshfree technique in \cite{dehghan2012method}. Moreover in \cite{saadatmandi2010numerical}, Saadatmandi and Dehghan  approximate the solution in terms of shifted Chebyshev polynomials to get the numerical solution. In \cite{bulbul2011taylor} authors proposed Taylor matrix method for the numerical solution of telegraph equation. Also, in \cite{heydary} authors implement Chebyshev wavelet method to solve one-dimensional telegraph equation numerically. In the progress of solution of telegraph equation, various meshless techniques have been developed. Dehgan and Mohebbi established a higher-order implicit collocation method in \cite{dehghan2009high}. An unconditionally stable implicit scheme has been presented by Mohanty et al.\cite{mohanty2002unconditionally}. Ding and Zhang \cite{ding2009new} developed a fourth-order compact finite difference scheme.  Jiwari et al. \cite{jiwari2012differential} proposed a numerical technique based on polynomial differential quadrature method (PDQM). Mittal and Bhatia proposed cubic B-spline collocation method in \cite{rachna} and differential quadrature method based on modified B-spline with space discretization in \cite{mittal_2014} . Recently, an operational matrix approach based on Bernoulli polynomials is proposed by S. Singh et al.\cite{singh2018application}. The methods based on operational matrices have proved to be very effective. The main advantage of using operational matrices is the sparsity of the operational matrices. In the numerical analysis, operational matrices based on approximation technique provides a powerful technique for approximating solutions of partial differential equations which is arising from mathematical modelling (for instant see \cite{singh2017numerical,patel_2017_1,patel_2017_2,patel_2017_3,patel_2017_4}). The motivation and philosophy behind the operational matrix approach is that it have some characteristic as follows:
\begin{itemize}
\item It reduces singularities from the proposed mathematical problems in an easy way.

\item It does not only simplify the proposed problem but also speed up the computation.

\item It is transformed the PDEs into the algebraic system.

\item The method is computer-oriented, thus solving higher-order PDEs becomes a matter of dimension increasing.

\item The solution is convergent, even though the size of increment may be large.
\end{itemize}
The basic idea of an operational matrix technique is as follows:
\begin{itemize}
\item The unknown function or its derivatives with respect to time(or space) in the given
PDEs are approximated by linear combinations of the orthonormal
basis functions and truncating them up to optimal levels. 

\item In this article, the operational matrix approximation the proposed problem converted into simple
algebraic equations whose solutions can be obtained using Sylvester's approach that gives approximate solutions for PDEs.
\end{itemize}

Nowadays, ''Wavelets'' has been an exceptionally famous subject of discussions in numerous zones of logical and designing angles. Some view wavelets as another reason for speaking to capacities, some consider it as a system for time-recurrence investigation, and others consider it another numerical subject. Obviously, every one of them are right, since ''wavelets'' is a flexible apparatus with rich scientific substance and extraordinary potential for applications. In any case, as this subject is still amidst quick improvement, it is certainly too soon to give a unified presentation. The subject of wavelets has had a spot in the core of engineering, science, and mathematics. Wavelet  is an energizing new technique for taking care of troublesome issues in engineering, mathematics, and physics, with present day applications as differing as wave proliferation, data compression, image processing, signal processing, computer graphics, the detection of aircraft and submarines,  medical image technology and pattern recognition. As the contribution of wavelets by Chebyshev , Bernoulli, and Legendre wavelets (for instant see \cite{singh2017numerical,patel_2017_1,patel_2017_2,patel_2017_3,patel_2017_4})based solution of partial differential equations picked up momentum in attractive way.  Favourable circumstances of Wavelets bases over operational matrix strategy have prompted enormous application in science and engineering. The exact solution proves the accuracy and efficiency of wavelets operational matrix methods with the good agreement of mathematical results. Likewise, the wavelets operational matrix strategy is simple, efficient and delivers extremely precise numerical outcomes in impressively modest number of basis function and hence reduces computational exertion. Additionally, the technique is easy to apply for multidimensional problems.

So, firstly we transform equations (\ref{1})-(\ref{3}) into its equivalent construction of integro-PDEs which consists of both initial and boundary conditions and therefore, we can be solved in a more suitable matter. Then by using the operational matrices of integration and differentiation of Legendre wavelets together with completeness of these wavelets, the integro-PDEs reduces to the system of algebraic Sylvester equation. Hence, we can achieve solution of (\ref{1})-(\ref{3}) in terms of Legendre wavelets on solving the Sylvester equation by generalised biconjugate gradient stabilized method (BICGSTAB) (i.e. robust Krylov subspace iterative method \cite{singh2018application}).
 
The outline of  this article is as follows. In section 2, we explain an introduction to the Multidimensional Legendre wavelet, function approximation and convergence of approximations. In section 3, we constructed operational matrices based on Legendre wavelet for multidimensional functional approximation with Kronecker multidimensional. In section 4, we proposed numerical method for solution. We employ some literature problems for showing the ability of the new technique in the current investigation in section 5. Finally a conclusion is given in section 6.

\section{Preliminaries: Construction of basis functions and their properties}
In this section, properties of Legendre wavelets and their associative operational matrices will be reviewed. Further, using the Kronecker multiplication, we will extend one dimensional operational matrix to multi dimensions. Legendre wavelets appeared in several different tendencies of engineering and sciences. One can point out many applications of Legendre wavelet in the numerical solution of partial differential equations (PDEs) like: Schr$\ddot{o}$dinger equation \cite{sadeghi_2019}, Poisson equation in \cite{heydari_2013}, etc. In this part of the article, properties of Legendre wavelets will be discussed that have essential roles in the Legendre wavelet operational matrix method (for instant see \cite{patel_2017_2}). 
 
 \subsection{Legendre wavelet \cite{patel_2017_2}}
 Wavelets constitute a family of functions constructed from dilation and translation of single function called the mother wavelet. When the dilation parameter a and the translation parameter b vary continuously, we have the following family of continuous wavelets:
 $$\Psi_{a,b}(\eta)={\lvert a\rvert}^{-\frac{1}{2}}\Psi\left( \frac{\eta-b}{a}\right)  ,a,b\in \textbf{R}, a\neq{0}. $$\\
 If the parameter a and b are restricted to the discrete values as $a=a_{0}^{-k,} , b=nb_{0}a_{0}^{-k} , a_{0}>1 , b_{0}>0 $, n and k are positive integers, from the above Eq. we have the following family of discrete wavelets :
 $$ \Psi_{k,n}(\eta)=\lvert a_{0}\rvert^{-\frac{k}{2}}\Psi{(a_{0}^{k}\eta-nb_{0})}, $$
 where, $\Psi_{k,n}(\eta)$ form a wavelet basis for $L^{2}(\mathbf{R})$.In particular, when $a_{0}=2$ and $b_{0}=1$ then $\Psi_{k,n}(\eta$ form an orthonormal basis. \\

 Legendre wavelet $\Psi_{n,m}(\eta)=\Psi(k,\hat{n},m,\eta)$ have four arguments $\hat{n}=2n-1, n=1,2,...,2^{k-1},k\in \textbf{Z}^{+},$ $m$ is the order of Legendre polynomials and $\eta$ is normalized time. They are defined on $[0,1)$ as follows (see  \cite{patel_2017_2}):
 \begin{equation}\label{4}
  \Psi_{n,m}(\eta)=\Psi(k,\hat{n},m,\eta)= \left\{ \begin{array}{ll}
         \sqrt{m+\frac{1}{2}}~2^{\frac{k}{2}}p_{m}(2^{k}\eta-\hat{n}), & \mbox{if $\frac{\hat{n}-1} {2^{k}}\leq \eta< \frac{\hat{n}+1} {2^{k}}$};\\
         0, & \mbox{if otherwise}.\end{array} \right.
 \end{equation}
 where, the coefficient $\sqrt{m+\frac{1}{2}}$ is for orthonormality.
 \subsection{Multi-dimensional Legendre wavelet}
 Two-dimensional Legendre wavelet can be expressed as product of one-dimensional Legendre wavelet as follows:
 \begin{equation}\label{5}
     \Psi_{n,m,n^{'},m^{'}}(\eta,\xi)=\left\{ \begin{array}{l}\Psi_{n,m}(\eta) \Psi_{n^{'},m^{'}}(\xi)~~,\mathrm{if}~
     \frac{\hat{n}-1} {2^{k}}\leq \eta< \frac{\hat{n}+1} {2^{k}},$$\\
     ~~~~~~~~~~~~~~~~~~~~~~~~~~~$$\frac{\hat{n'}-1} {2^{k'}}\leq \xi< \frac{\hat{n'}+1} {2^{k'}},\\  ~~0~~~~~~~~~~~~~~~~~~~~~,\vspace{0.3cm} \mathrm{ otherwise ,} \end{array}\right.
 \end{equation}
 
  Three-dimensional Legendre wavelet can be expressed as product of one-dimensional Legendre wavelet as follows:
  \begin{equation}\label{6}
      \Psi_{n,m,n^{'},m^{'}}(\eta,\xi,t)=\left\{ \begin{array}{l}\Psi_{n,m}(\eta) \Psi_{n^{'},m^{'}}(\xi)\Psi_{n^{''},m^{''}}(t)~~,\mathrm{if}~
      \frac{\hat{n}-1} {2^{k}}\leq \eta< \frac{\hat{n}+1} {2^{k}},$$\\
      ~~~~~~~~~~~~~~~~~~~~~~~~~~~$$\frac{\hat{n'}-1} {2^{k'}}\leq \xi< \frac{\hat{n'}+1} {2^{k'}},\\ 
       ~~~~~~~~~~~~~~~~~~~~~~~~~~~$$\frac{\hat{n''}-1} {2^{k''}}\leq t< \frac{\hat{n''}+1} {2^{k''}},\\
       ~~0~~~~~~~~~~~~~~~~~~~~~,\vspace{0.3cm} \mathrm{ otherwise ,} \end{array}\right.
  \end{equation}
 where,
 \begin{equation} \label{7}
 \Psi_{n,m}(\eta)=\sqrt{m+\frac{1}{2}} 2^{\frac{k} {2}}p_{m}(2^{k}\eta-\hat{n}),\\
 \end{equation}
 \begin{equation} \label{8}
 \Psi_{n^{'},m^{'}}(\xi)=\sqrt{m^{'}+\frac{1}{2}} 2^{\frac{k^{'}} {2}} p_{m^{'}}(2^{k^{'}}\xi-\hat{n'}),\\
 \end{equation}
  \begin{equation} \label{9}
  \Psi_{n^{''},m^{''}}(t)=\sqrt{m^{''}+\frac{1}{2}} 2^{\frac{k^{''}} {2}} p_{m^{''}}(2^{k^{''}}t-\hat{n''}),\\
  \end{equation}
 and $$m=0,1,2,...,M-1, M\in Z^{+}\cup \left\lbrace 0\right\rbrace,$$
 $$m^{'}=0,1,2,...,M^{'}-1,M'\in Z^{+}\cup \left\lbrace 0\right\rbrace,$$ 
 $$m^{''}=0,1,2,...,M^{''}-1,M''\in Z^{+}\cup \left\lbrace 0\right\rbrace,$$
 $$\hat{n}=2n-1, \hat{n'}=2n'-1, \hat{n''}=2n''-1,$$
 $$ n=1,2,3,...,2^{k-1}, n^{'}=1,2,3,...,2^{k^{'}-1}, n^{''}=1,2,3,...,2^{k^{''}-1}$$
 here also $Z^{+}$ is positive integer and $p_{m}$, $p_{m^{'}}$ and $p_{m^{''}}$ are Legendre polynomial of order $m, m'$ and $m^{''}$ respectively which are defined over the interval $[0,1]$ and also three-dimensional Legendre wavelet are orthonormal set over $\Omega=[0,1]\times [0,1]\times [0,1]$.\\

\subsection{\textbf{Function approximation}	}
Suppose that $f(\eta)$ is an arbitrary function in $L^{2}([0,1))$, then it can be approximated as follows:	
	\begin{equation}\label{10}
	f(\eta)=\sum_{n = 1}^{\infty}\sum_{m = 1}^{\infty} f_{nm} \Psi_{nm}(\eta) 
	\end{equation}
If the infinite series (\ref{10}) is truncated for $m=M-1$, then approximation of (\ref{10}) can be represented as in the following form
	\begin{equation}\label{11}
	f(\eta)=\sum_{n = 1}^{2^{k-1}}\sum_{m = 1}^{M-1} f_{nm} \Psi_{nm}(\eta)=F_{1}^{T}\Psi(\eta).
	\end{equation}
	where $\Psi(\eta)=[\Psi_{10}(\eta),\Psi_{11}(\eta),...,\Psi_{1(M-1)},...,\Psi_{2^{k-1}0},\Psi_{2^{k-1}1},..., \Psi_{2^{k-1}(M-1)}(\eta)]^{T}$ and $F_{7}^{T}$ is $1 \times (2^{k-1}(M-1))$ vector given by:
	\begin{equation*}
	F_{7}^{T}=\begin{bmatrix}
	\int_{0}^{1}f(\eta) \Psi_{10}(\eta) d\eta\\
	\int_{0}^{1}f(\eta) \Psi_{11}(\eta)d\eta\\
	\vdots\\
	\int_{0}^{1}f(\eta) \Psi_{1(M-1)}(\eta)d\eta\\
	\vdots\\
	\int_{0}^{1}f(\eta) \Psi_{2^{K-1}0}(\eta)d\eta\\
	\int_{0}^{1}f(\eta) \Psi_{2^{K-1}1}(\eta)d\eta\\
	\vdots\\
	\int_{0}^{1}f(\eta) \Psi_{2^{K-1}(M-1)}(\eta)d\eta\\
	\end{bmatrix}^{T}.
	\end{equation*}
Now generalised (\ref{11}) series as follows
	\begin{equation}\label{12}
	f(\eta,\xi)=\sum_{n = 1}^{2^{k-1}}\sum_{m = 1}^{M-1}\sum_{n' = 1}^{2^{k'-1}}\sum_{m' = 1}^{M'-1} f_{nmn'm'} \Psi_{nmn'm'}(\eta,\xi)=F_{8}^{T}\Psi(\eta,\xi),
	\end{equation} 
Where, $\Psi(\eta,\xi)=\Psi(\eta)\otimes \Psi(\xi)$ (For the numerical solution, we used the concept of Kronecker product $(\otimes)$ \cite{patel_2017_2})is $2^{k-1}2^{k^{'}-1}MM^{'}\times 1$, vector given as follows: \\
\begin{equation} \label{13}
\begin{split}
\Psi=[\Psi_{1010},\cdots,\Psi_{101(M'-1)},\Psi_{1020}\cdots,\Psi_{102(M'-1)},\cdots,
\Psi_{102^{k^{'}-1}0},\cdots,~~~~~~~~~~~~~~~~~~~~~~\\
\Psi_{102^{k^{'}-1}(M'-1)},\cdots,
\Psi_{1(M-1)10},\cdots,\Psi_{1(M-1)1(M'-1)},
\Psi_{1(M-1)20},\cdots,\Psi_{1(M-1)2(M'-1)},~~~~~~~~~~~~~\\
\Psi_{1(M-1)2^{k^{'}-1}0}\cdots,
\Psi_{1(M-1)2^{k^{'}-1}(M'-1)},
\Psi_{2010},\cdots,\Psi_{201(M'-1)},\Psi_{2020},
\cdots,~~~~~~~~~~~~~~~~~~~~~~~\\
\Psi_{202(M'-1)},\cdots,
\Psi_{202^{k^{'}-1}0},\cdots, \Psi_{202^{k^{'}-1}(M'-1)},\cdots,
\Psi_{2(M-1)10},\cdots,\Psi_{2(M-1)1(M'-1)},~~~~~~~~~~~~\\
\Psi_{2(M-1)20},\cdots,
\Psi_{2(M-1)2(M'-1)},\Psi_{2(M-1)2^{k^{'}-1}0}\cdots,
\Psi_{2(M-1)2^{k^{'}-1}(M'-1)},
\cdots,~~~~~~~~~~~~~~~~~~~~~~~~~\\
\Psi_{2^{k-1}010},\cdots ,\Psi_{2^{k-1}01(M'-1)},\Psi_{2^{k-1}020},\cdots,
 \Psi_{2^{k-1}(M-1)2^{k^{'}-1}(M'-1)}]^{T},~~~~~~~~~~~~~~~~~~~~~~~~~~~~~~\\
\end{split}
\end{equation}
and
	\begin{equation} \label{14}
	F_{8}^{T}=\begin{bmatrix}
	\int_{0}^{1}\int_{0}^{1}f(\eta,\xi) \Psi_{1010}(\eta) d\eta d\xi\\
	\int_{0}^{1}\int_{0}^{1} f(\eta,\xi) \Psi_{1011}(\eta)d\eta d\xi\\
	\vdots\\
	\int_{0}^{1}\int_{0}^{1}f(\eta,\xi) \Psi_{101(M'-1)}(\eta)d\eta d\xi\\
	\vdots\\
	\int_{0}^{1}\int_{0}^{1}f(\eta,\xi) \Psi_{111(M'-1)}(\eta) d\eta d\xi\\
	\int_{0}^{1}\int_{0}^{1} f(\eta,\xi) \Psi_{121(M'-1)}(\eta)d\eta d\xi\\
	\vdots\\
	\int_{0}^{1}\int_{0}^{1}f(\eta,\xi) \Psi_{2^{k-1}02^{k-1}0}(\eta)d\eta d\xi\\
	\int_{0}^{1}\int_{0}^{1}f(\eta,\xi) \Psi_{2^{k-1}02^{k-1}1}(\eta)d\eta d\xi\\
	\vdots\\
	\int_{0}^{1}\int_{0}^{1}f(\eta,\xi) \Psi_{2^{k-1}02^{k'-1}(M'-1)}(\eta)d\eta d\xi\\
	\int_{0}^{1}\int_{0}^{1}f(\eta,\xi) \Psi_{2^{k-1}12^{k'-1}(M'-1)}(\eta)d\eta d\xi\\
	\vdots\\
	\int_{0}^{1}\int_{0}^{1}f(\eta,\xi) \Psi_{2^{k-1}(M-1)2^{k'-1}(M'-1)}(\eta)d\eta d\xi\\
	\end{bmatrix}^{T}.
	\end{equation}
and
\begin{equation}\label{15}
f(\eta,\xi,t)=\sum_{n = 1}^{2^{k-1}}\sum_{m = 1}^{M-1}\sum_{n' = 1}^{2^{k'-1}}\sum_{m' = 1}^{M'-1}\sum_{n'' = 1}^{2^{k''-1}}\sum_{m'' = 1}^{M''-1} f_{nmn'm'n''m''} \Psi_{nmn'm'n''m''}(\eta)=F_{9}^{T}\Psi(\eta,\xi,t)
\end{equation} 
Where, $\Psi(\eta,\xi,t)=\Psi(\eta)\otimes \Psi(\xi) \otimes \Psi(t)$ are $2^{k-1}2^{k'-1}2^{k''-1}MM'M''\times 1$, vector given as follows: \\

	\begin{equation} \label{16}
	F_{9}^{T}=\begin{bmatrix}
	\int_{0}^{1}\int_{0}^{1}\int_{0}^{1}f(\eta,\xi,t) \Psi_{101010}(\eta,\xi,t) d\eta d\xi dt\\
	\int_{0}^{1}\int_{0}^{1}\int_{0}^{1} f(\eta,\xi,t) \Psi_{101011}(\eta,\xi,t)d\eta d\xi dt\\
	\vdots\\
	\int_{0}^{1}\int_{0}^{1}\int_{0}^{1}f(\eta,\xi,t) \Psi_{10101(M''-1)}(\eta,\xi,t)d\eta d\xi dt\\
		\vdots\\
	\int_{0}^{1}\int_{0}^{1}\int_{0}^{1}f(\eta,\xi,t) \Psi_{2(M-1)2(M'-1)2(M''-1)}(\eta,\xi,t)d\eta d\xi dt\\
			\vdots\\
	\int_{0}^{1}\int_{0}^{1}\int_{0}^{1}f(\eta,\xi,t) \Psi_{2^{k-1}02^{k'-1}02^{k''-1}0}(\eta,\xi,t)d\eta d\xi dt\\
	\int_{0}^{1}\int_{0}^{1}\int_{0}^{1}f(\eta,\xi,t) \Psi_{2^{k-1}02^{k'-1}02^{k''-1}1}(\eta,\xi,t)d\eta d\xi dt\\
	\vdots\\
	\int_{0}^{1}\int_{0}^{1}\int_{0}^{1}f(\eta,\xi,t) \Psi_{2^{k-1}02^{k'-1}02^{k''-1}(M''-1)}(\eta,\xi,t)d\eta d\xi dt\\
	\int_{0}^{1}\int_{0}^{1}\int_{0}^{1}f(\eta,\xi,t) \Psi_{2^{k-1}02^{k'-1}12^{k''-1}(M''-1)}(\eta,\xi,t)d\eta d\xi dt\\
		\vdots\\
	\int_{0}^{1}\int_{0}^{1}\int_{0}^{1}f(\eta,\xi,t) \Psi_{2^{k-1}02^{k'-1}(M'-1)2^{k''-1}(M''-1)}(\eta,\xi,t)d\eta d\xi dt\\
	\int_{0}^{1}\int_{0}^{1}\int_{0}^{1}f(\eta,\xi,t) \Psi_{2^{k-1}12^{k'-1}(M'-1)2^{k''-1}(M''-1)}(\eta,\xi,t)d\eta d\xi dt\\
			\vdots\\
	\int_{0}^{1}\int_{0}^{1}\int_{0}^{1}f(\eta,\xi,t) \Psi_{2^{k-1}(M-1)2^{k'-1}(M'-1)2^{k''-1}(M''-1)}(\eta,\xi,t)d\eta d\xi dt
	\end{bmatrix}^{T}.
	\end{equation}
\begin{theorem} The series 
$\sum_{\substack{n=1}}^{\substack{\infty}}\sum_{m=0}^{\infty}\sum_{n'=1}^{\infty}\sum_{m'=0}^{\infty}\sum_{n''=1}^{\infty}\sum_{m''=0}^{\infty}f_{nmn'm'n''m''}\Psi_{nmn'm'n''m''}(\eta,\xi,t)$, where $\Psi(\eta,\xi,t)$ is three-dimensional Legendre wavelets which is defined in (\ref{6}), is uniformly converges to a continuous function $f(\eta,\xi,t)$.
\end{theorem}
\begin{proof}
Let $L^{2}(\lambda_{1})$ be the Hilbert space and $\Psi(\eta,\xi,t)$ is defined in (\ref{6}) forms an orthonormal basis. So for fixed $k,k'$ and $k''$\\
\begin{equation} \label{17}
f(\eta,\xi,t)=\sum_{m=0}^{M-1}\sum_{m'=0}^{M'-1}\sum_{m''=0}^{M''-1}f_{mm'm''} \Psi_{mm'm''}(\eta,\xi,t),
\end{equation} 
where, 
\begin{equation}\label{18}
\begin{split}
f_{mm'm''}=\int_{0}^{1}\int_{0}^{1}\int_{0}^{1}f(\eta,\xi,t) \Psi_{mm'm''}(\eta,\xi,t)d\eta d\xi dt,\\
          =\left\langle f(\eta,\xi,t), \Psi_{mm'm''}(\eta,\xi,t)\right\rangle.
\end{split}
\end{equation}
Now truncate series (\ref{18}) upto $N$ level as follows 
\begin{equation}\label{19}
f(\eta,\xi,t)=\sum_{m=0}^{N}\sum_{m'=0}^{N}\sum_{m''=0}^{N}f_{mm'm''} \Psi_{mm'm''}(\eta,\xi,t)=S_{N} (\mathrm{say}).
\end{equation} 
Now, 

\begin{equation}\label{20}
\begin{split}
\left\langle f(\eta,\xi,t), S_{N} \right\rangle=\left\langle f(\eta,\xi,t), \sum_{m=0}^{N}\sum_{m'=0}^{N}\sum_{m''=0}^{N}f_{mm'm''} \Psi_{mm'm''}(\eta,\xi,t)\right\rangle  \\
                                               =\sum_{m=0}^{N}\sum_{m'=0}^{N}\sum_{m''=0}^{N}\bar{f}_{mm'm''} \left\langle f(\eta,\xi,t), \Psi_{mm'm''}(\eta,\xi,t)\right\rangle. 
\end{split}
\end{equation}
From (\ref{19}) and (\ref{20}) 
\begin{equation}\label{21}
\begin{split}
\left\langle f(\eta,\xi,t), S_{N} \right\rangle=\sum_{m=0}^{N}\sum_{m'=0}^{N}\sum_{m''=0}^{N}\bar{f}_{mm'm''} f_{mm'm''},\\
                                               =\sum_{m=0}^{N}\sum_{m'=0}^{N}\sum_{m''=0}^{N} \lvert f_{mm'm''}\rvert^{2}.
\end{split}
\end{equation}
Now, we claim that
\begin{equation}\label{22}
\begin{split}
\left\langle f(\eta,\xi,t), S_{N} \right\rangle=\sum_{m=0}^{N}\sum_{m'=0}^{N}\sum_{m''=0}^{N} \lvert f_{mm'm''}\rvert^{2} \mathrm{for} N\geq N'.
\end{split}
\end{equation}
Since
\begin{equation}\label{23}
\begin{split}
\lVert S_{N}-S_{N'} \rVert^{2}=\lVert \sum_{m=0}^{N}\sum_{m'=0}^{N}\sum_{m''=0}^{N}f_{mm'm''} \Psi_{mm'm''}(\eta,\xi,t)-\sum_{i=0}^{N}\sum_{j=0}^{N}\sum_{k=0}^{N}f_{ijk} \Psi_{ijk}(\eta,\xi,t) \rVert^{2},\\
                              =\lVert \sum_{m=N'+1}^{N}\sum_{m'=N'+1}^{N}\sum_{m''=N'+1}^{N}f_{mm'm''} \Psi_{mm'm''}(\eta,\xi,t)\rVert^{2},\\
                              =\left\langle \sum_{m=N'+1}^{N}\sum_{m'=N'+1}^{N}\sum_{m''=N'+1}^{N}f_{mm'm''} \Psi_{mm'm''}(\eta,\xi,t),\sum_{i=N'+1}^{N}\sum_{j=N'+1}^{N}\sum_{k=N'+1}^{N}f_{ijk} \Psi_{ijk}(\eta,\xi,t) \right\rangle, \\
                              =\sum_{m=N'+1}^{N}\sum_{m'=N'+1}^{N}\sum_{m''=N'+1}^{N}\sum_{i=N'+1}^{N}\sum_{j=N'+1}^{N}\sum_{k=N'+1}^{N}f_{mm'm''}\bar{f}_{ijk} \left\langle \Psi_{mm'm''}(\eta,\xi,t),\Psi_{ijk} \right\rangle,\\
                              =\sum_{m=N'+1}^{N}\sum_{m'=N'+1}^{N}\sum_{m''=N'+1}^{N}\lvert f_{mm'm''} \rvert^{2}.
\end{split}
\end{equation}
Now using Bessel's inequality series (\ref{23}) is convergent and hence
\begin{equation}
\lVert S_{N}-S_{N'} \rVert \rightarrow~~ 0~~ \mathrm{as}~~~ N,N' \rightarrow \infty. 
\end{equation}\label{24}
So $\left\langle S_{N}\right\rangle$ is Cauchy sequence in $L^{2}(\lambda_{1})$,
i.e. $S_{N}\rightarrow s \mathrm{(say)}$.
Now,
\begin{equation}\label{25}
\begin{split}
\left\langle s-f(\eta,\xi,t),\Psi_{mm'm''}(\eta,\xi,t)\right\rangle =\left\langle s,\Psi_{mm'm''}(\eta,\xi,t)\right\rangle-\left\langle f(\eta,\xi,t),\Psi_{mm'm''}(\eta,\xi,t)\right\rangle,\\
                                 =\left\langle \lim_{N\rightarrow\infty}S_{N},\Psi_{mm'm''}(\eta,\xi,t)\right\rangle-\left\langle f(\eta,\xi,t),\Psi_{mm'm''}(\eta,\xi,t)\right\rangle,\\
                                 =\lim_{N\rightarrow\infty}\left\langle S_{N},\Psi_{mm'm''}(\eta,\xi,t)\right\rangle-f_{mm'm''},\\
                                 =\lim_{N\rightarrow\infty}\left\langle  \sum_{m=N'+1}^{N}\sum_{m'=N'+1}^{N}\sum_{m''=N'+1}^{N}f_{mm'm''} \Psi_{mm'm''}(\eta,\xi,t),\Psi_{mm'm''}(\eta,\xi,t),\Psi_{ijk}(\eta,\xi,t)\right\rangle-f_{mm'm''},\\
                                 =\lim_{N\rightarrow\infty}f_{mm'm''} \left\langle \Psi_{mm'm''}(\eta,\xi,t),\Psi_{ijk}(\eta,\xi,t)\right\rangle-f_{mm'm''},\\
                                 =f_{mm'm''}-f_{mm'm''}~~~~~~\mathrm{(using ~orthonormality ~of ~Legendre~ wavelets)},\\
                                 =0.\hspace{11cm}~~~~~
\end{split}
\end{equation}
Hence $s=f(\eta,\xi,t)$.\\
Thus the series  $\sum_{\substack{n=1}}^{\substack{\infty}}\sum_{m=0}^{\infty}\sum_{n'=1}^{\infty}\sum_{m'=0}^{\infty}\sum_{n''=1}^{\infty}\sum_{m''=0}^{\infty}f_{nmn'm'n''m''}\Psi_{nmn'm'n''m''}(\eta,\xi,t)$ converge uniformly to $f(\eta,\xi,t)$.\\
Hence the theorem.

\end{proof}
\section{Operational Matrices}
Let $\Psi(t)=[\Psi_{0}(t),\Psi_{1}(t),...,\Psi_{N}(t)]^{T}$ be the basis functions. Then,
\begin{equation}
\frac{d}{dt}\begin{bmatrix}
\Psi_{0}(t)\\
\Psi_{1}(t)\\
\vdots\\
\Psi_{N}(t)
\end{bmatrix}
\approx
\begin{bmatrix}
F & 0 & 0 & \cdots & 0\\
0 & F & 0 & \cdots & 0\\
\vdots & \vdots & \ddots & \ddots & \vdots\\
0 & 0 & 0 & \cdots   & F
\end{bmatrix}_{2^{k-1}M \times 2^{k-1}M}
\begin{bmatrix}
\Psi_{0}(t)\\
\Psi_{1}(t)\\
\vdots\\
\Psi_{N}(t)
\end{bmatrix}=D_{t} \Psi^{T}(t),
\end{equation}
 in which $F$ is $M\times M$ matrix with entries
$$
 F_{r,s}= \left\{ \begin{array}{ll}
         2^{k}\sqrt{(2r-1)(2s-1)}, & \mbox{if $r=2,3,...M, s=1,2,...,r-1, (r+s) \mathrm{odd}$};\\
         0, & \mbox{if otherwise}.\end{array} \right.
$$
And 
$$
\int_{0}^{t}\begin{bmatrix}
\Psi_{0}(t')\\
\Psi_{1}(t')\\
\vdots\\
\Psi_{N}(t')
\end{bmatrix}dt'
\approx 
$$
$$
\begin{bmatrix}
1 & \sqrt{3} & 0 & 0 & \cdots & 0 & 0 & 0\\
-\frac{\sqrt{3}}{3}& 0 &\frac{\sqrt{3}}{3\sqrt{5}}& 0 & \cdots & 0 & 0 & 0\\
0 & -\frac{\sqrt{5}}{5\sqrt{3}}& 0 &\frac{\sqrt{5}}{5\sqrt{7}}& 0 & \cdots & 0 & 0 & 0\\
\vdots & \vdots & \vdots & \vdots & \ddots & \vdots & \vdots & \vdots\\
0 & 0 & 0 & 0 & \cdots & \frac{\sqrt{2M-3}}{(2M-3)\sqrt{2M-5}} & 0 & \frac{\sqrt{2M-3}}{(2M-3)\sqrt{2M-1}}\\
0 & 0 & 0 & 0 & \cdots & 0 & \frac{\sqrt{2M-1}}{(2M-1)\sqrt{2M-3}} & 0   
\end{bmatrix}
\begin{bmatrix}
\Psi_{0}(t)\\
\Psi_{1}(t)\\
\vdots\\
\Psi_{N}(t)
\end{bmatrix}=I_{t} \Psi^{T}(t)
$$
Here, $I_{t}$ is ${ 2^{k-1}M \times 2^{k-1}M}$ matrix.

The approximation given in section (2) for numerical method can be extended for the higher dimension. Two variables functions, namely $f_{3}(\xi,t),f_{4}(\xi,t),f_{5}(\eta,t)$ and $f_{6}(\eta,t)$ in $L^2([0,1]\times[0,1])$ can be approximated as:
\begin{equation}\label{26}
\left\{ \begin{array}{l}
f_{3}(\xi,t)\approx\sum_{n'=1}^{2^{k'-1}}\sum_{m'=0}^{M'-1}\sum_{n''=1}^{2^{k''-1}}\sum_{m''=0}^{M''-1}f_{n'm'n''m''}^{3}\Psi_{m}(\xi)\Psi_{n}(t)=\Psi^{T}(\xi)F_{3}\Psi(t),
\vspace{0.2cm}\\
f_{4}(\xi,t)\approx\sum_{n'=1}^{2^{k'-1}}\sum_{m'=0}^{M'-1}\sum_{n''=1}^{2^{k''-1}}\sum_{m''=0}^{M''-1}f_{n'm'n''m''}^{4}\Psi_{m}(\xi)\Psi_{n}(t)=\Psi^{T}(\xi)F_{4}\Psi(t),
\vspace{0.2cm}\\
f_{5}(\eta,t)\approx\sum_{n=1}^{2^{k-1}}\sum_{m=0}^{M-1}\sum_{n''=1}^{2^{k''-1}}\sum_{m''=0}^{M''-1}f_{nmn''m''}^{5}\Psi_{m}(\eta)\Psi_{n}(t)=\Psi^{T}(\eta)F_{5}\Psi(t),
\vspace{0.2cm}\\
f_{6}(\eta,t)\approx\sum_{n=1}^{2^{k-1}}\sum_{m=0}^{M-1}\sum_{n''=1}^{2^{k''-1}}\sum_{m''=0}^{M''-1}f_{nmn''m''}^{6}\Psi_{m}(\eta)\Psi_{n}(t)=\Psi^{T}(\eta)F_{6}\Psi(t),
\end{array}\right.
\end{equation}
where    

$$F_{3}=
\begin{bmatrix}
f_{1010}^{3} & f_{1011}^{3} & \cdots & f_{101(M''-1)}^{3} & f_{111(M''-1)}^{3} & \cdots f_{1(M'-1)1(M''-1)}^{3}\\
f_{2020}^{3} & f_{2021}^{3} & \cdots & f_{202(M''-1)}^{3} & f_{212(M''-1)}^{3} & \cdots f_{2(M'-1)2(M''-1)}^{3}\\
\vdots & \vdots & \ddots & \vdots & \vdots & \vdots\\
f_{2^{k'-1}02^{k''-1}0}^{3} & f_{2^{k'-1}02^{k''-1}1}^{3} & \cdots & f_{2^{k'-1}02^{k''-1}(M''-1)}^{3} & f_{2^{k'-1}12^{k''-1}(M''-1)}^{3} & \cdots f_{2^{k'-1}(M'-1)2^{k'-1}(M''-1)}^{3}\\
\end{bmatrix},
$$
$$
F_{4}=
\begin{bmatrix}
f_{1010}^{4} & f_{1011}^{4} & \cdots & f_{101(M''-1)}^{4} & f_{111(M''-1)}^{4} & \cdots f_{1(M'-1)1(M''-1)}^{4}\\
f_{2020}^{4} & f_{2021}^{4} & \cdots & f_{202(M''-1)}^{4} & f_{212(M''-1)}^{4} & \cdots f_{2(M'-1)2(M''-1)}^{4}\\
\vdots & \vdots & \ddots & \vdots & \vdots & \vdots\\
f_{2^{k'-1}02^{k''-1}0}^{4} & f_{2^{k'-1}02^{k''-1}1}^{4} & \cdots & f_{2^{k'-1}02^{k''-1}(M''-1)}^{4} & f_{2^{k'-1}12^{k''-1}(M''-1)}^{4} & \cdots f_{2^{k'-1}(M'-1)2^{k'-1}(M''-1)}^{4}
\end{bmatrix}
$$
$$
F_{5}=
\begin{bmatrix}
f_{1010}^{5} & f_{1011}^{5} & \cdots & f_{101(M''-1)}^{5} & f_{111(M''-1)}^{5} & \cdots f_{1(M'-1)1(M''-1)}^{5}\\
f_{2020}^{5} & f_{2021}^{5} & \cdots & f_{202(M''-1)}^{5} & f_{212(M''-1)}^{5} & \cdots f_{2(M'-1)2(M''-1)}^{5}\\
\vdots & \vdots & \ddots & \vdots & \vdots & \vdots\\
f_{2^{k'-1}02^{k''-1}0}^{5} & f_{2^{k'-1}02^{k''-1}1}^{5} & \cdots & f_{2^{k'-1}02^{k''-1}(M''-1)}^{5} & f_{2^{k'-1}12^{k''-1}(M''-1)}^{5} & \cdots f_{2^{k'-1}(M'-1)2^{k'-1}(M''-1)}^{5}
\end{bmatrix},
$$
$$
F_{4}=
\begin{bmatrix}
f_{1010}^{6} & f_{1011}^{6} & \cdots & f_{101(M''-1)}^{6} & f_{111(M''-1)}^{6} & \cdots f_{1(M'-1)1(M''-1)}^{6}\\
f_{2020}^{6} & f_{2021}^{6} & \cdots & f_{202(M''-1)}^{6} & f_{212(M''-1)}^{6} & \cdots f_{2(M'-1)2(M''-1)}^{6}\\
\vdots & \vdots & \ddots & \vdots & \vdots & \vdots\\
f_{2^{k'-1}02^{k''-1}0}^{6} & f_{2^{k'-1}02^{k''-1}1}^{6} & \cdots & f_{2^{k'-1}02^{k''-1}(M''-1)}^{6} & f_{2^{k'-1}12^{k''-1}(M''-1)}^{6} & \cdots f_{2^{k'-1}(M'-1)2^{k'-1}(M''-1)}^{6}
\end{bmatrix}.
$$
where,
$$
f_{n'm'n''m''}^{3}=\int_{0}^{1}\int_{0}^{1}f_{3}(\xi,t) \Psi(\xi,t) d\xi dt,
$$
$$
f_{n'm'n''m''}^{4}=\int_{0}^{1}\int_{0}^{1}f_{4}(\xi,t) \Psi(\xi,t) d\xi dt,
$$
$$
f_{n'm'n''m''}^{5}=\int_{0}^{1}\int_{0}^{1}f_{5}(\eta,t) \Psi(\eta,t) d\eta dt,
$$
$$
f_{n'm'n''m''}^{6}=\int_{0}^{1}\int_{0}^{1}f_{6}(\eta,t) \Psi(\eta,t)d\eta dt.
$$
\begin{Lemma} Let
$\lambda_{1}=\int_{0}^{1}\Psi^{T}(\eta)d\eta$  and $\lambda_{1}{I_{L}}=\lambda_{1} I_{L}^{T}$ are the $1 \times 2^{k-1}M  $ vectors, then the following relations hold
\begin{itemize}
\item[(i)]
$\Psi^{T}(\xi)=\Psi^{T}(\eta,\xi)P$,
\item[(ii)]
$\eta \Psi^{T}(\xi)=\Psi^{T}(\eta,\xi)\overline{P},\hspace{-0.3cm}$
\item[(iii)]
$\eta(\lambda_{1}{I_{\eta}}\otimes I \Psi^{T}(\xi))=\Psi^{T}(\eta,\xi)Q,\hspace{-1.5cm}$
\item[(iv)]
$\xi( \Psi^{T}(\eta)I\otimes \lambda_{1}{I_{\eta}})=\Psi^{T}(\eta,\xi)\overline{Q},\hspace{-1.5cm}$
\item[(v)]
$\Psi^{T}(\eta)=\Psi^{T}(\eta,\xi)R,$
\item[(vi)]
$\xi \Psi^{T}(\eta)=\Psi^{T}(\eta,\xi)\overline{R}.\hspace{-0.2cm}$
\end{itemize}
\end{Lemma}
\begin{proof}

Let $\Psi(\eta)=[\Psi_{10}(\eta),\Psi_{11}(\eta),...,\Psi_{1(M-1)}(\eta),...,\Psi_{2^{k-1}0}(\eta), \Psi_{2^{k-1}0}(\eta),...,\Psi_{2^{k-1}(M-1)}(\eta)]^T$ be the one dimensional Legendre wavelet basis of $L^2[0,1]$. Then $\eta \Psi^{T}(\xi)$ can be written in the following form \\
\begin{equation} \label{27}
\eta \Psi^{T}(\xi)=\eta \otimes \Psi^{T}(\xi).
\end{equation} \label{28}
Let $\eta=g(\eta) \in L^{2}[0,1]$. Then $g(\eta)$ can be approximated in terms of Legendre wavelet basis functions as 
$$g(\eta)=\eta \approx \sum_{n=1}^{2^{k-1}}\sum_{m=0}^{M-1} g_{nm} \Psi_{nm}(\eta)=A^{T} \Psi(\eta)$$
where, $ A = [a_{10}, a_{11},...,a_{1(M-1)},..., a_{2^{k-1}0}, a_{2^{k-1}1},...,a_{2^{k-1}(M-1)}]^T$ and the coefficients of $A$ are calculated by the formula
$$a_{i}=\int_{0}^{1} \eta \Psi_{i}(\eta) d\eta$$
hence from Eq.$(18)$
\begin{equation} \label{29}
\eta \Psi^{T}(\xi)=\Psi^{T}(\eta) A \otimes \Psi^{T}(\xi) =\left( \Psi^{T}(\eta) \otimes \Psi^{T}(\xi) \right) \left(  A \otimes I_{N+1}\right)  = \Psi^{T}(\eta,\xi)\overline{P},
\end{equation}
where, $$\overline{P}=A \otimes I.$$

Similarly, we can write $(iii)$ as $$\eta\left( \lambda_{1}{I_{\eta}}\otimes I \Psi^{T}(\xi)\right)=\left(\eta \lambda_{1}{I_{\eta}}\right)\otimes I \Psi^{T}(\xi) .$$
Since, $\eta \lambda_{1}{I_{L}}$ is $1 \times 2^{k-1}M$ vector and each element of this vector is a function of $\eta$
so we can make similar argument as given in the proof of $(ii)$.\\ 
Hence,
\begin{equation}\label{30}
\begin{split}
\left(\eta \lambda_{1}{I_L}\right)\otimes I \Psi^{T}(\xi)=\left( \Psi^{T}(\eta) B^{T}\right)  \otimes I \Psi^{T}(\xi),\hspace{3cm}\\
                                                  =\left( \Psi^{T}(\eta) \otimes \Psi^{T}(\xi) \right) \left(  B^{T} \otimes I\right),\hspace{1.6cm}\\
                                                  =\Psi^{T}(\eta,\xi)P,\hspace{4.7cm}
\end{split}
\end{equation}
where, $P=\left( B^{T}\otimes I_{N+1}\right) $
and $$B=\begin{bmatrix}
c_{10} & c_{11} & \cdots & c_{1(M-1)}\\
c_{20} & c_{21} & \cdots & c_{2(M-1)}\\
\vdots & \vdots & \ddots & \vdots\\
c_{2^{k-1}0} & c_{2^{k-1}1} & \cdots & c_{2^{k-1}(M-1)}
\end{bmatrix}=[c_{nm}]_{2^{k-1}M \times 2^{k-1}M}.$$
The coefficients  $c_{nm}$ are calculated as follows $$c_{nm}= \int_{0}^{1}\left[ \left( \eta \lambda_{1}{I_\eta}\right) \Psi_{nm}(\eta) \right] d\eta.$$
\textbf{Note:} (a)~All other matrices $\overline{Q}, R$ and $\overline{R}$  are calculated in the similar manner.\\[5pt]
~~~~~(b)~ If we take $\eta=1$ then the proof of $(i)$ is same as $(ii)$. Similarly, the proof of $(iv),(v)$ and $(vi)$ are same as $(i),(ii)$ and $(iii)$ respectively.
\end{proof}

\begin{Lemma}
If $D_{\xi}=I_{2^{k'-1}M' \times 2^{k'-1}M'}\otimes D_{2^{k'-1}M' \times 2^{k'-1}M'},
I_{\xi}=I_{2^{k'-1}M' \times 2^{k'-1}M'}\otimes I_{L,2^{k'-1}M' \times 2^{k'-1}M'}$ and $I_{\eta}=I_{L,2^{k-1}M \times 2^{k-1}M}\otimes I_{2^{k'-1}M' \times 2^{k'-1}M'},$ where
$I$ denotes the identity matrix  then we have\\
\begin{itemize}
\item[(i)]
$\frac{\partial \Psi(\eta,\xi)}{\partial
\xi}=\Psi_{\xi}(\eta,\xi)=D_{\xi}\Psi(\eta,\xi),$
\item[(ii)]
$\frac{\partial^{2} \Psi(\eta,\xi)}{\partial
\xi^{2}}=\Psi_{\xi\xi}(\eta,\xi)=(D_{\xi})^{2}\Psi(\eta,\xi),$
\item[(iii)]
$\int_{0}^{\xi}\Psi(\eta,\xi')d\xi'\approx I_{\xi}\Psi(\eta,\xi),$
\item[(iv)]
$\int_{0}^{\xi}\int_{0}^{\xi'}\Psi(\eta,\xi'')d\xi''d\xi'\approx
(I_{\xi})^{2}\Psi(\eta,\xi),$
\item[(v)]
$\int_{0}^{\eta}\Psi(\eta',\xi)d\eta'\approx I_{\eta}\Psi(\eta,\xi),$
\item[(vi)]
$\int_{0}^{\eta}\int_{0}^{\eta'}\Psi(\eta'',\xi)d\eta''d\eta'\approx
(I_{\eta})^{2}\Psi(\eta,\xi).$
\item[(vii)] $\frac{\partial \Psi(t)}{\partial
t}=\Psi_{t}(t)=D_{t}\Psi(t),$
\item[(viii)]$\frac{\partial^{2}\Psi(t)}{\partial t^{2}}=\Psi_{tt}(t)=D_{t}^{2}\Psi(t),$
\end{itemize}
\end{Lemma}
\begin{proof}
In section 3, we can see the proof which are summarized here.
\end{proof}
\section{Numerical Method of Solution}
To find numerical solution of Eqs.(\ref{1})-(\ref{3}), firstly, we will use all the initial and boundary conditions on (\ref{1}) then (\ref{1}) will 
convert into partial integro-differential equation (PIDE). For this purpose we rewrite Eq.(\ref{1}) as follows
\begin{equation}\label{31}
\frac{\partial^{2}\Phi(\eta,\xi,t)}{\partial
\eta^{2}}=\frac{\partial^{2} \Phi(\eta,\xi,t)}{\partial
t^{2}}+2\lambda_{1}\frac{\partial \Phi(\eta,\xi,t)}{\partial
t}+\lambda_{2}^{2}\Phi(\eta,\xi,t)-\frac{\partial^{2}\Phi(\eta,\xi,t)}{\partial
\xi^{2}}-F(\eta,\xi,t).
\end{equation}
Integrating Eq.(\ref{31}) in the interval $[0,\eta]$, we obtain
\begin{equation}\label{32}
\begin{split}
\Phi_{\eta}(\eta,\xi,t)=\Phi_{\eta}(0,\xi,t)+\int_{0}^{\eta}(\Phi_{tt}(\eta',\xi,t)+2\lambda_{1}
\Phi_{t}(\eta',\xi,t)+\lambda_{2}^{2}\Phi(\eta',\xi,t)\\
-\Phi_{\xi\xi}(\eta',\xi,t)-F(\eta',\xi,t))d\eta'.
\end{split}
\end{equation}
further integrating Eq.(\ref{32}) in the interval $[0,\eta]$, we get
\begin{equation}\label{33}
\begin{split}
\Phi(\eta,\xi,t)=\underbrace{\Phi(0,\xi,t)}_{=f_{3}(\xi,t)}+\eta \Phi_{\eta}(0,\xi,t)+\int_{0}^{\eta}\int_{0}^{\eta'}(\Phi_{tt}(\eta'',\xi,t)+2\lambda_{1}
\Phi_{t}(\eta'',\xi,t)\\+\lambda_{2}^{2}\Phi(\eta'',\xi,t)
-\Phi_{\xi\xi}(\eta'',\xi,t)-F(\eta'',\xi,t))d\eta''d\eta'.
\end{split}
\end{equation}
Put $\eta=1$ in Eq.(\ref{33}), we get
\begin{equation}\label{34}
\begin{split}
\Phi_{\eta}(0,\xi,t)=\underbrace{\Phi(1,\xi,t)}_{=f_{4}(\xi,t)}-f_{1}(\xi,t)-\int_{0}^{1}\int_{0}^{\eta}( \Phi_{tt}(\eta',\xi,t)+2\lambda_{1}\Phi_{t}(\eta',\xi,t)\\
+\lambda_{2}^{2}\Phi(\eta',\xi,t)-\Phi_{\xi\xi}(\eta',\xi,t)-F(\eta',\xi,t)) d\eta'd\eta.
\end{split}
\end{equation}
Substituting the value of $\Phi_{\eta}(0,\xi,t)$ from Eq.(\ref{34}) to (\ref{33}), we get
\begin{equation}\label{35}
\begin{split}
\Phi(\eta,\xi,t)=\eta f_{4}(\xi,t)+(1-\eta)f_{3}(\xi,t)-\eta\int_{0}^{1}\int_{0}^{\eta}( \Phi_{tt}(\eta',\xi,t)+2\lambda_{1} \Phi_{t}(\eta',\xi,t)\\
+\lambda_{2}^{2}\Phi(\eta',\xi,t)-\Phi_{\xi\xi}(\eta',\xi,t)-F(\eta',\xi,t))d\eta'd\eta
+\int_{0}^{\eta}\int_{0}^{\eta'}(\Phi_{tt}(\eta'',\xi,t)\\
+2\lambda_{1}\Phi_{t}(\eta'',\xi,t)+\lambda_{2}^{2}\Phi(\eta'',\xi,t)-\Phi_{\xi\xi}(\eta'',\xi,t)-F(\eta'',\xi,t)) d\eta''d\eta'.
\end{split}
\end{equation}
 To use boundary conditions in $\xi$, we integrate $\Phi_{\xi\xi}(\eta,\xi,t)$ from 0 to $\xi$, we obtain
 
\begin{equation}\label{36}
\int_{0}^{\xi}\Phi_{\xi\xi}(\eta,\xi',t)d\xi
=\Phi_{\xi}(\eta,\xi,t)-\Phi_{\xi}(\eta,0,t),
\end{equation} 
 further integrating Eq.(\ref{36}) from 0 to $\xi$,  we get 
$$ 
 \int_{0}^{\xi}\int_{0}^{\xi'}\Phi_{\xi\xi}(\eta,\xi'',t)d\xi''d\xi'=\Phi(\eta,\xi,t)-\Phi(\eta,0,t)-\xi \Phi_{\xi}(\eta,0,t),
$$
or
 \begin{equation} \label{37}
\begin{split}
  \Phi(\eta,\xi,t)=f_{5}(\eta,t)+\xi \Phi_{\xi}(\eta,0,t)+\int_{0}^{\xi}\int_{0}^{\xi'}\Phi_{\xi\xi}(\eta,\xi'',t)d\xi''d\xi'.
\end{split}
\end{equation} 
 To use the condition $\Phi(\eta,1,t)=f_{6}(\eta,t)$, put $\xi=1$ in Eq.(\ref{37}), we get
 
 \begin{equation} \label{38}
 \Phi_{\xi}(\eta,0,t)=f_{6}(\eta,t)- f_{5}(\eta,t) -\int_{0}^{1}\int_{0}^{\xi}\Phi_{\xi\xi}(\eta,\xi',t)d\xi'd\xi,
  \end{equation} 
 substitute the value of $\Phi_{\xi}(\eta,0,t)$ obtained by Eq.(\ref{38}) into (\ref{37}), we obtain
 
\begin{equation}\label{39}
\begin{split}
\Phi(\eta,\xi,t)=\xi
f_{6}(\eta,t)+(1-\xi)f_{5}(\eta,t)-\xi\int_{0}^{1}\int_{0}^{\xi}\Phi_{\xi\xi}(\eta,\xi',t)d\xi'd\xi\\
+\int_{0}^{\xi}\int_{0}^{\xi'}\Phi_{\xi\xi}(\eta,\xi'',t)d\xi''d\xi'.
\end{split}
\end{equation}
Now, our goal is to use the initial conditions (\ref{2}). For this purpose, we use same process which is used to use the boundary conditions in $\xi$ i.e integrate $\Phi_{tt}(\eta,\xi,t)$ two times from $0$ to $t$ and obtain
\begin{equation*}
\Phi(\eta,\xi,t)=f_{1}(\eta,\xi)+t
f_{2}(\eta,\xi)+\int_{0}^{t}\int_{0}^{t'}\Phi_{tt}(\eta,\xi,t'')dt''dt'.
\end{equation*}
The above Eq. can be rewritten as
\begin{equation}\label{40}
\Phi(\eta,\xi,t)=f(\eta,\xi,t)+\int_{0}^{t}\int_{0}^{t'}\Phi_{tt}(\eta,\xi,t'')dt''dt'
\end{equation}
where, $f(\eta,\xi,t)=f_{1}(\eta,\xi)+t f_{2}(\eta,\xi)$. 

Since, we can obtain the original Eq. (\ref{1}) with boundary conditions $\Phi(0,\xi,t)=f_3(\xi,t)$ and $\Phi(1,\xi,t)=f_4(\xi,t)$ by differentiating two times Eqs. (\ref{35}) and replacing $\eta$ by $0$ and $1$ in Eq. (\ref{35}) respectively. Including this boundary conditions $\Phi(\eta,0,t)=h_1(\eta,t)$, $\Phi(\eta,1,t)=h_2(\eta,t)$ and initial conditions $\Phi(\eta,\xi,0)=f_1(\eta,\xi)$, $\Phi_{t}(\eta,\xi,0)=f_2(\eta,\xi)$ received by replacing $\xi$ by $0$ and $1$ in Eq. (\ref{39}) and $t$ by 0 and 1 in Eq. (\ref{40}) respectively. Hence, the Eqs. (\ref{35}),(\ref{39}) and (\ref{40}) are the equivalent formulation of the proposed problem (\ref{1})-(\ref{3}). To solve these PIDE, all the known and unknown functions are approximated in terms of basis functions. 
The known function of three variables $F(\eta,\xi,t)$ can be approximated  in terms of basis functions as follows
\begin{equation}\label{41}
F(\eta,\xi,t)\approx
\sum_{n = 1}^{2^{k-1}}\sum_{m = 1}^{M-1}\sum_{n' = 1}^{2^{k'-1}}\sum_{m' = 1}^{M'-1}\sum_{n'' = 1}^{2^{k''-1}}\sum_{m'' = 1}^{M''-1} F_{nmn'm'n''m''}\Psi_{nm}(\eta)\Psi_{n'm'}(\xi)\Psi_{n''m''}(t)=\Psi^{T}(\eta,\xi)F\Psi(t),
\end{equation}
where $\Psi^{T}(\eta,\xi)=\Psi^{T}(\eta)\otimes \Psi^{T}(\xi)$ and $F$ is
 matrix calculated in similar manner as $F$ in section 3. Similarly, we approximate $f$ as follows 
\begin{equation}\label{42}
f(\eta,\xi,t)\approx \Psi^{T}(\eta,\xi)F\Psi(t).
\end{equation}
Similarly, the unknown function $\Psi(\eta,\xi,t)$ is approximated in terms of basis functions as
\begin{equation}\label{43}
\Phi(\eta,\xi,t)\approx
\sum_{n = 1}^{2^{k-1}}\sum_{m = 1}^{M-1}\sum_{n' = 1}^{2^{k'-1}}\sum_{m' = 1}^{M'-1}\sum_{n'' = 1}^{2^{k''-1}}\sum_{m'' = 1}^{M''-1}\Phi_{nmn'm'n''m''}\Psi_{nm}(\eta)\Psi_{n'm'}(\xi)\Psi_{n''m''}(t)=\Psi^{T}(\eta,\xi) \Phi_{apx} \Psi(t),
\end{equation}
where
\begin{equation*}
\Phi_{apx}=
\begin{bmatrix}
\Psi_{101010} & \Psi_{202020} & \cdots & \Psi_{2^{k-1}02^{k'-1}02^{k''-1}0}\\
\Psi_{101011} & \Psi_{202020} & \cdots & \Psi_{2^{k-1}02^{k'-1}02^{k''-1}0}\\
\vdots & \vdots & \ddots & \vdots\\
\Psi_{10101(M''-1)} & \Psi_{20202(M''-1)} & \cdots & \Psi_{2^{k-1}02^{k'-1}02^{k''-1}(M''-1)}\\
\Psi_{10111(M''-1)} & \Psi_{20212(M''-1)} & \cdots & \Psi_{2^{k-1}02^{k'-1}12^{k''-1}(M''-1)}\\
\vdots & \vdots & \ddots & \vdots\\
\Psi_{101(M'-1)1(M''-1)} & \Psi_{202(M'-1)2(M''-1)} & \cdots & \Psi_{2^{k-1}02^{k'-1}(M'-1)2^{k''-1}(M''-1)}\\
\Psi_{111(M'-1)1(M''-1)} & \Psi_{212(M'-1)2(M''-1)} & \cdots & \Psi_{2^{k-1}12^{k'-1}(M'-1)2^{k''-1}(M''-1)}\\
\vdots & \vdots & \ddots & \vdots\\
\Psi_{1(M-1)1(M'-1)1(M''-1)} & \Psi_{2(M-1)2(M'-1)2(M''-1)} & \cdots & \Psi_{2^{k-1}(M-1)2^{k'-1}(M'-1)2^{k''-1}(M''-1)}
\end{bmatrix}^{T}
\end{equation*}
where, $\Phi_{apx}$ is $({2^{k-1}M 2^{k'-1}M') \times 2^{k''-1}M''}$ matrix. It is to be noted that, we have to calculate the unknown matrix $\Phi_{apx}$. \par

Now, substitute the approximated value of $F(\eta,\xi,t),\Phi(\eta,\xi,t),f_{1}(\xi,t)$ and $f_{2}(\xi,t)$ which are given by Eqs. (\ref{41})-(\ref{43}) and (\ref{26}) in the right hand side of Eq. (\ref{35}). After using the operational matrices of differentiation and integration and Lemma 2, Eq. (\ref{35}) reduces into the following form
\begin{equation*}
\begin{split}
\Phi(\eta,\xi,t)\approx\eta(\Psi^{T}(\xi)F_{4}\Psi(t))+(1-\eta)(\Psi^{T}(\xi)F_{3}\Psi(t))\hspace{10cm}\\
-\eta\int_{0}^{1}\int_{0}^{\eta}\Psi^{T}(\eta',\xi)\left[ \Phi_{apx} D_{t}^{2}+2\lambda_{1} \Phi_{apx} D_{t}+\lambda_{2}^{2}\Phi_{apx}-(D_{\xi}^{T})^{2}\Phi_{apx}-F\right] \Psi(t)d\eta'd\eta\hspace{3cm}\\
+\int_{0}^{\eta}\int_{0}^{\eta'}\Psi^{T}(\eta'',\xi)\left[ \Phi_{apx} D_{t}^{2}+2\lambda_{1} \Phi_{apx} D_{t}+\lambda_{2}^{2} \Phi_{apx}-(D_{\xi}^{T})^{2} \Phi_{apx}-F\right] \Psi(t)d\eta''d\eta',\hspace{3cm}
\end{split}
\end{equation*}
or,
\begin{equation}\label{44}
\begin{split}
\Phi(\eta,\xi,t)\approx \eta \Psi^{T}(\xi)(F_{4}-F_{3})\Psi(t)+\Psi^{T}(\xi)F_{3}\Psi(t)\hspace{10cm}\\
-\eta(\lambda_{1}{I_\eta}\otimes \Psi^{T}(\xi))(\Phi_{apx} D_{t}^{2}+2\lambda_{1} \Phi_{apx} D_t+\lambda_{2}^{2} \Phi_{apx}-(D_{\xi}^{T})^{2}\Phi_{apx}-F)\Psi(t)\hspace{4cm}\\
+\Psi^{T}(\eta,\xi)(I^{T}_{\eta})^{2}(\Phi_{apx} D_{t}^{2}+2\lambda_{1} \Phi_{apx} D_{t}+\lambda_{2}^{2}\Phi_{apx}-(D_{\xi}^{T})^{2}\Phi_{apx}-F)\Psi(t).\hspace{4cm}
\end{split}
\end{equation}
Using (i), (ii) and (iii) of Lemma 1, we can rewrite Eq. (\ref{44}) as
\begin{equation}\label{45}
\begin{split}
\Phi(\eta,\xi,t)=\Psi^{T}(\eta,\xi)\overline{P}(F_{4}-F_{3})\Psi(t)+\Psi^{T}(\eta,\xi)P F_{3}\Psi(t)\\
-\Psi^{T}(\eta,\xi)Q(\Phi_{apx} D_{t}^{2}+2\lambda_{1} \Phi_{apx} D_{t}+\lambda_{2}^{2}\Phi_{apx}-(D_{\xi}^{T})^{2}\Phi_{apx}-F)\Psi(t)\\
+\Psi^{T}(\eta,\xi)(I^{T}_{\eta})^{2}(\Phi_{apx} D_{t}^{2}+2\lambda_{1} \Phi_{apx} D_{t}+\lambda_{2}^{2}\Phi_{apx}-(D_{\xi}^{T})^{2}\Phi_{apx}-F)\Psi(t).
\end{split}
\end{equation}
Now Eq. (\ref{45}) can be rewritten as
\begin{equation}\label{46}
\Phi(\eta,\xi,t)\approx \Psi^{T}(\eta,\xi)X\Psi(t),
\end{equation}
in which
\begin{equation*}
X=\overline{P}(F_{4}-F_{3})+P F_{1}+((I^{T}_{\eta})^{2}-Q)(\Phi_{apx}D_{t}^{2}+2\lambda_{1} \Phi_{apx} D_{t}+\lambda_{2}^{2}\Phi_{apx}-(D_{\xi}^{T})^{2}\Phi_{apx}-F).
\end{equation*}

Now, substitute the approximated value of $\Phi(\eta,\xi,t)$ from Eq.(\ref{46})
and $f_{3}(\eta,t),f_{4}(\eta,t)$ from Eq. (\ref{26}) in Eq. (\ref{39}) and
using the operational matrices and Lemma 2, we get
\begin{equation}\label{47}
\begin{split}
\Phi(\eta,\xi,t)\approx \xi(\Psi^{T}(\eta)F_{6}\Psi(t))+(1-\xi)\Psi^{T}(\eta)F_{5}\Psi(t)\\
-\xi\int_{0}^{1}\int_{0}^{\xi}\Psi^{T}(\eta,\xi)(D_{\xi}^{T})^{2}X\Psi(t)d\xi'd\xi\\
+\int_{0}^{\xi}\int_{0}^{\xi'}\Psi^{T}(\eta,\xi)(D_{\xi}^{T})^{2}X\Psi(t)d\xi''d\xi'.
\end{split}
\end{equation}
We can rewrite the above Eq. as
\begin{equation}\label{48}
\begin{split}
\Phi(\eta,\xi,t)\approx \xi \Psi^{T}(\eta)(F_{6}-F_{5})\Psi(t)+\Psi^{T}(\eta)F_{5}\Psi(t)\\
-\xi(\Psi^{T}\otimes\lambda_{1}{I_\eta})(D_{\xi}^{T})^{2}X\Psi(t)+\Psi^{T}(\eta,\xi)(P^{T}_{\eta})^{2}(D_{\xi}^{T})^{2}X\Psi(t).
\end{split}
\end{equation}
By using (iv),(v) and (vi) of Lemma (1), above Eq. can be written as
\begin{equation}\label{49}
\Phi(\eta,\xi,t)\approx \Psi^{T}(\eta,\xi)Y\Psi(t),
\end{equation}
where
\begin{equation*}
Y=\overline{R}(F_{6}-F_{5})+R
F_{5}+((I^{T}_{\xi})^{2}-\overline{Q})(D_{\xi}^{T})^{2}X.
\end{equation*}

In the next step, Eq.(\ref{49})  can be taken as an approximation for $\Phi(\eta,\xi,t)$. Substituting the approximated values of $f(\eta,\xi,t)$, $\Phi(\eta,\xi,t)$   together with  operational
 matrices of differentiation and integration in the right hand side of Eq.(\ref{40}), we get
 \begin{equation*}
\Phi(\eta,\xi,t)\approx
\Psi^{T}(\eta,\xi)F\Psi(t)+\Psi^{T}(\eta,\xi)YD_{t}^{2}I_{t}^{2}\Psi(t).
 \end{equation*}
 On combining above approximation $\Phi(\eta,\xi,t)$ with Eq.(\ref{43}), we get 
 \begin{equation}\label{50}
 \Psi^{T}(\eta,\xi)\Phi_{apx} \Psi(t)= \Psi^{T}(\eta,\xi)F\Psi(t)+\Psi^{T}(\eta,\xi)\widetilde{Y}D_{t}^{2}I_{t}^{2}\Psi(t),
 \end{equation}
$$\widetilde{Y}=\overline{R}(F_{6}-F_{5})+R F_{5}+((I^{T}_{\xi})^{2}-\overline{Q})(D_{\xi}^{T})^{2}\widetilde{X}~~~~~~~~~~~~~~~~~~~~~~~~~~~~~~~~~~~~~~$$
\begin{equation*}
\widetilde{X}=\overline{P}(F_{4}-F_{3})+PF_{3}+((I^{T}_{\eta})^{2}-Q)(\Psi D_{t}^{2}+2\lambda_{1}\Phi_{apx} D_{t}+\lambda_{2}^{2} \Phi_{apx}-(D_{\xi}^{T})^{2} \Phi_{apx}-F).
\end{equation*}
Eq.(\ref{50}) can be reduced into the following matrix Eq.
\begin{equation}\label{51}
\Phi_{apx}+X_{1}\Phi_{apx} Y_{1}+X_{2} \Phi_{apx} Y_{2}=Z
\end{equation}
which is in the form of Sylvester Eq..\\
Where,
$$~~X_{1}=((I^{T}_{\xi})^{2}-\overline{Q})(D_{\xi}^{T})^{2}((I^{T}_{\eta})^{2}-Q),$$
$$ Y_{1}=-(D_{t}^{2}+2\lambda_{1} D_{t}+\lambda_{2}^{2})D_{t}^{2}I_{t}^{2},$$
$$~~~~~~~~~~~X_{2}=((I^{T}_{\xi})^{2}-\overline{Q})(D_{\xi}^{T})^{2}((I^{T}_{\eta})^{2}-Q)(D_{\xi}^{T})^{2},$$
$$Y_{2}=D_{t}^{2}I_{t}^{2}~~~~~~~~~~~~~~~~~~~~~~~~~~~~$$
 and
\begin{equation*}
\begin{split}
Z=F+R F_{5}D_{t}^{2}I_{t}^{2}+\overline{R}(F_{6}-F_{5})D_{t}^{2}I_{t}^{2}+((I^{T}_{\xi})^{2}-\overline{Q})(D_{\xi}^{T})^{2}\\
(\overline{P}(F_{4}-F_{3})+P F_{3}-((I^{T}_{\eta})^{2}-Q)F)D_{t}^{2}I_{t}^{2}.
\end{split}
\end{equation*}
Now Eq.(\ref{51}) can be re-written as follows
\begin{equation}\label{51_1}
[I+X_{1}(Y_{1}^{T}\otimes I)+X_{2}(Y_{2}^{T}\otimes I)]\Phi_{apx}=Z,
\end{equation}
where $I$ is the identity matrix.

Finally, we get the system of equation in the form of Eq.(\ref{51_1}) which is known as the Sylvester equation. Some authors proposed numerical scheme to solve Sylvester equation (for instant see \cite{bartels1972algorithm}). In this article Sylvester equation (\ref{51}) is solved for $\Phi_{apx}$ by robust Krylov subspace iterative method (i.e.generalized BICGSTAB, \cite{tohidi2014convergence}) and hence by Eq.(\ref{51_1}), we get the approximate solution of Eqs.(\ref{1})-(\ref{3}) in terms of Legendre wavelet at small number of basis functions.
\section{Numerical experiments}
In this section, we provide some numerical examples to analyze the applicability of proposed LWOMM for two dimensional hyperbolic telegraph equation (\ref{1}). For the application of proposed method, we have included three literature examples and examine the applicability and effectiveness (Numerical simulations have been done with the help of Matlab).  Since, the proposed method is a numerical technique based on the Legendre wavelet function, we can reach to the exact solutions if the solutions of the considered hyperbolic telegraph equations are in polynomial forms. Moreover, the presented method gain spectral accuracy for dealing with two-dimensional hyperbolic telegraph equation which has exact solutions in non-polynomial forms. We take $k=k'=k''=3, M=M'=M''=4,5,6,7$ \& $t=1$ also computed results have been shown in the Table in terms of relative errors, $l^{2}$ and $l^{\infty}$. The results are compared with Mittal and Bhatia \cite{mittal_2014} and found better accuracy by Legendre wavelet at small number of basis function. In order to illustrate the performance of the LWOMM, we have considered the maximum absolute error, which is denoted by $e_{n}$. Also, in some figures the history of error between the analytical and numerical solution, which is computed by the proposed LWOMM, is represented as $e_{n}=\Phi-\Phi_{apx}$. Further, to assess the performance of the method LWOMM we computed $l^{2}$ and $l^{\infty}$ norms errors.

\textbf{Example 1.}In this example, we consider the HTE Eq.(1) in the domain $(\eta,\xi,t)\in \Omega$ and
$F(\eta,\xi,t)=(-2\lambda_{1}+\lambda_{2}^2-1)e^{-t}\sinh \eta \sinh
\xi$. The initial and the Dirichlet boundary conditions are given by
\begin{equation}\label{52}
\left\{ \begin{array}{l} \Phi(\eta,\xi,0)=\sinh(\eta)\sinh(\xi), \vspace{0.2cm}\\
\Phi_{t}(\eta,\xi,0)=-\sinh(\eta)\sinh(\xi) \vspace{0.2cm}
\end{array}\right.
\end{equation}
\begin{equation}\label{53}
\left\{ \begin{array}{l} \Phi(0,\xi,t)=0,~\Phi(1,\xi,t)=e^{-t}\sinh(1)\sinh(\xi) \vspace{0.2cm}\\
\Phi(\eta,0,t)=0,~\Phi(\eta,1,t)=e^{-t}\sinh(\eta)\sinh(1)
\vspace{0.2cm}.
\end{array}\right.
\end{equation}
The exact solution is given by $\Phi(\eta,\xi,t)=e^{-t}\sinh(\eta)\sinh(\xi)$.
In this example, we have taken $\lambda_{1}=10, \lambda_{2}=5$ and solved for $N=4,5,6$ and $7$.
\captionsetup[figure]{labelfont={bf},labelformat={default},labelsep=period,name={Fig.}}
\begin{figure}[H]
	\vspace{-6cm}
	\includegraphics[width=\textwidth]{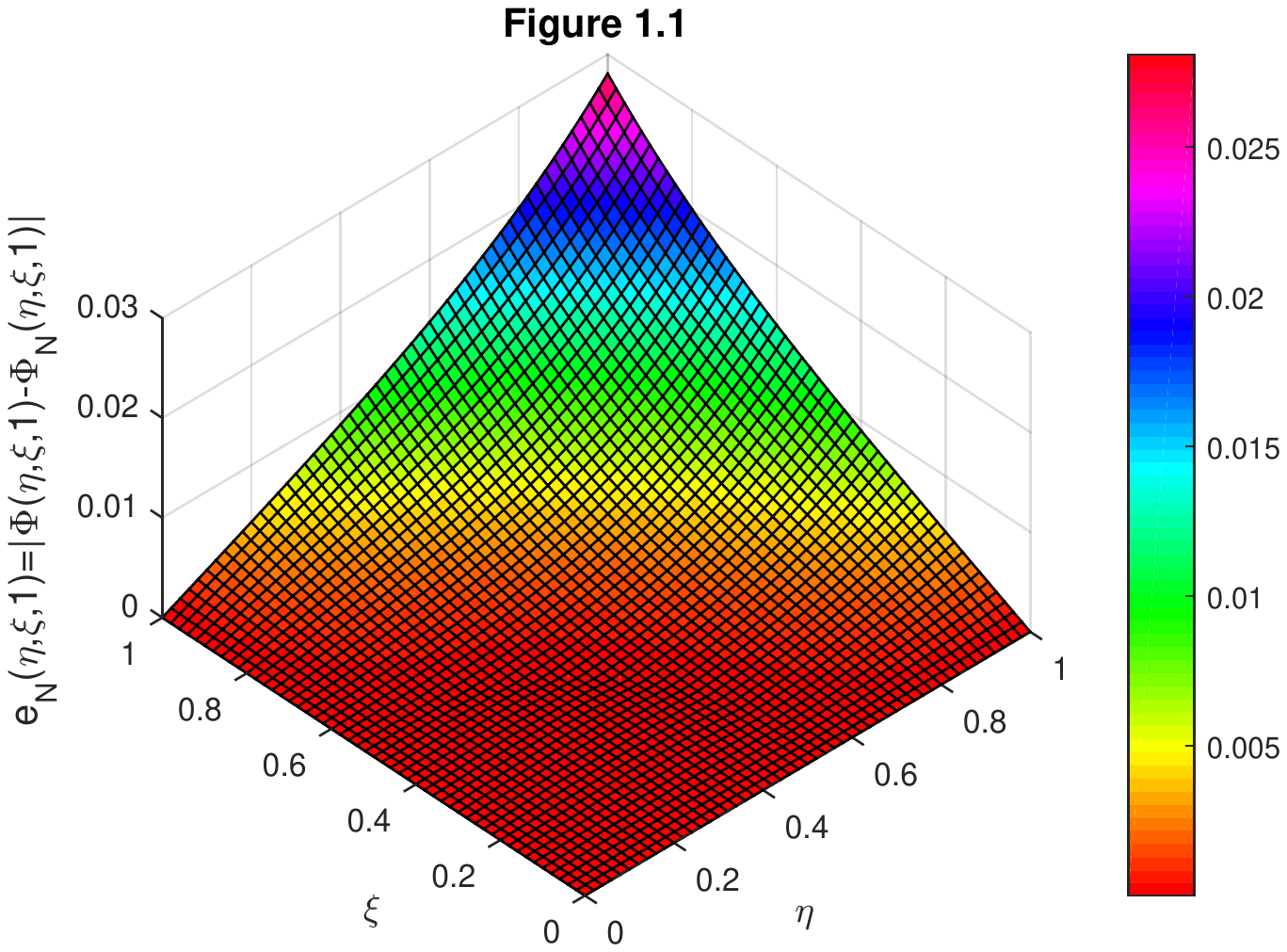}
	\vspace{-7cm}
	\caption{The spatio-temporal distribution of an error between the exact solution and its numerical solution utilising the LWOMM for Example 1 at $N=4$.} 	\label{fig1}
\end{figure}

\begin{figure}[H]
	\vspace{-7cm}
	\includegraphics[width=\textwidth]{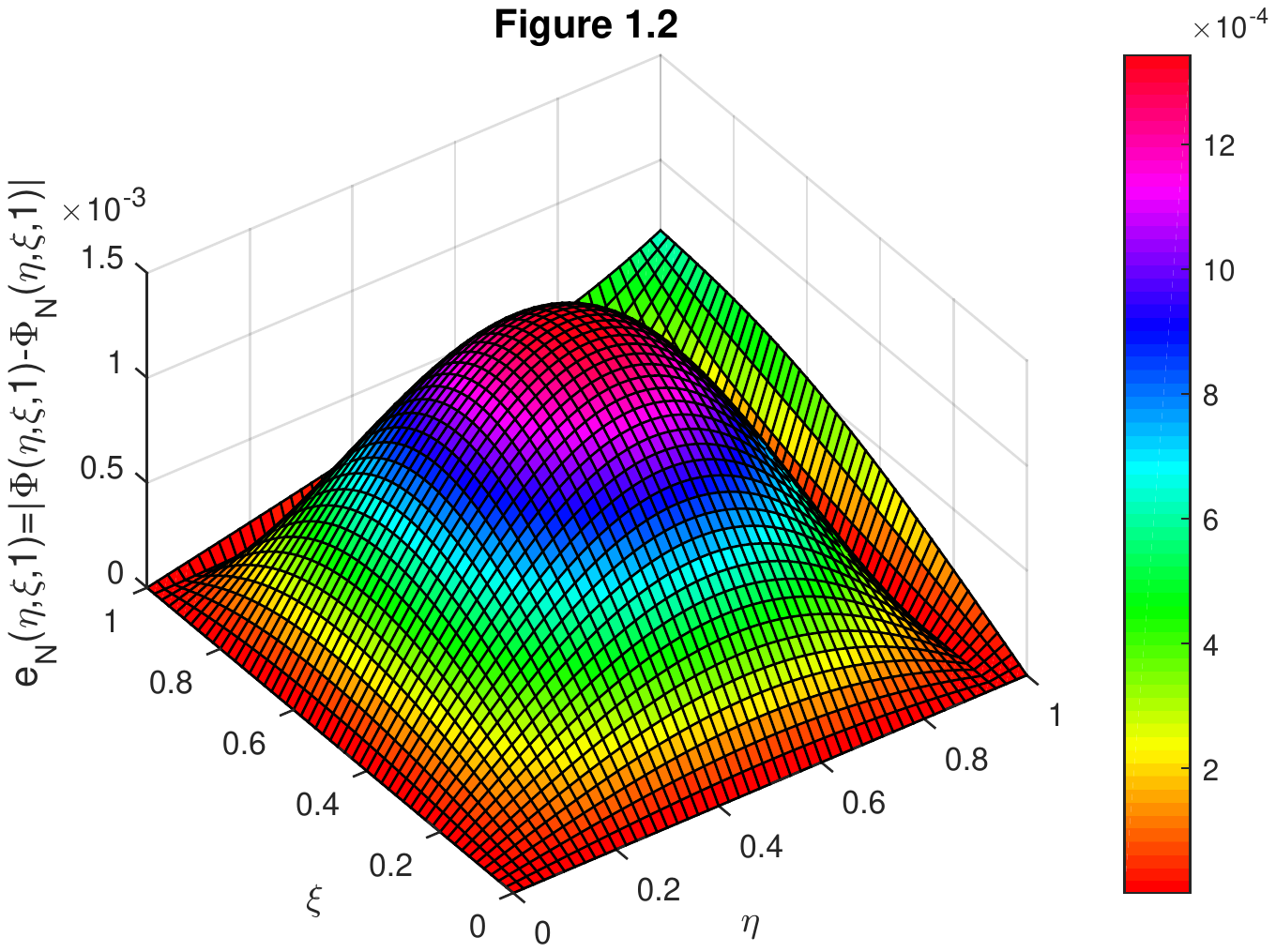}
	\vspace{-7cm}
	\caption{The spatio-temporal distribution of an error between the exact solution and its numerical solution utilising the LWOMM for Example 1 at $N=5$.} 
	\label{fig2}
\end{figure}

\begin{figure}[H]
	\vspace{-6cm}
	\includegraphics[width=\textwidth]{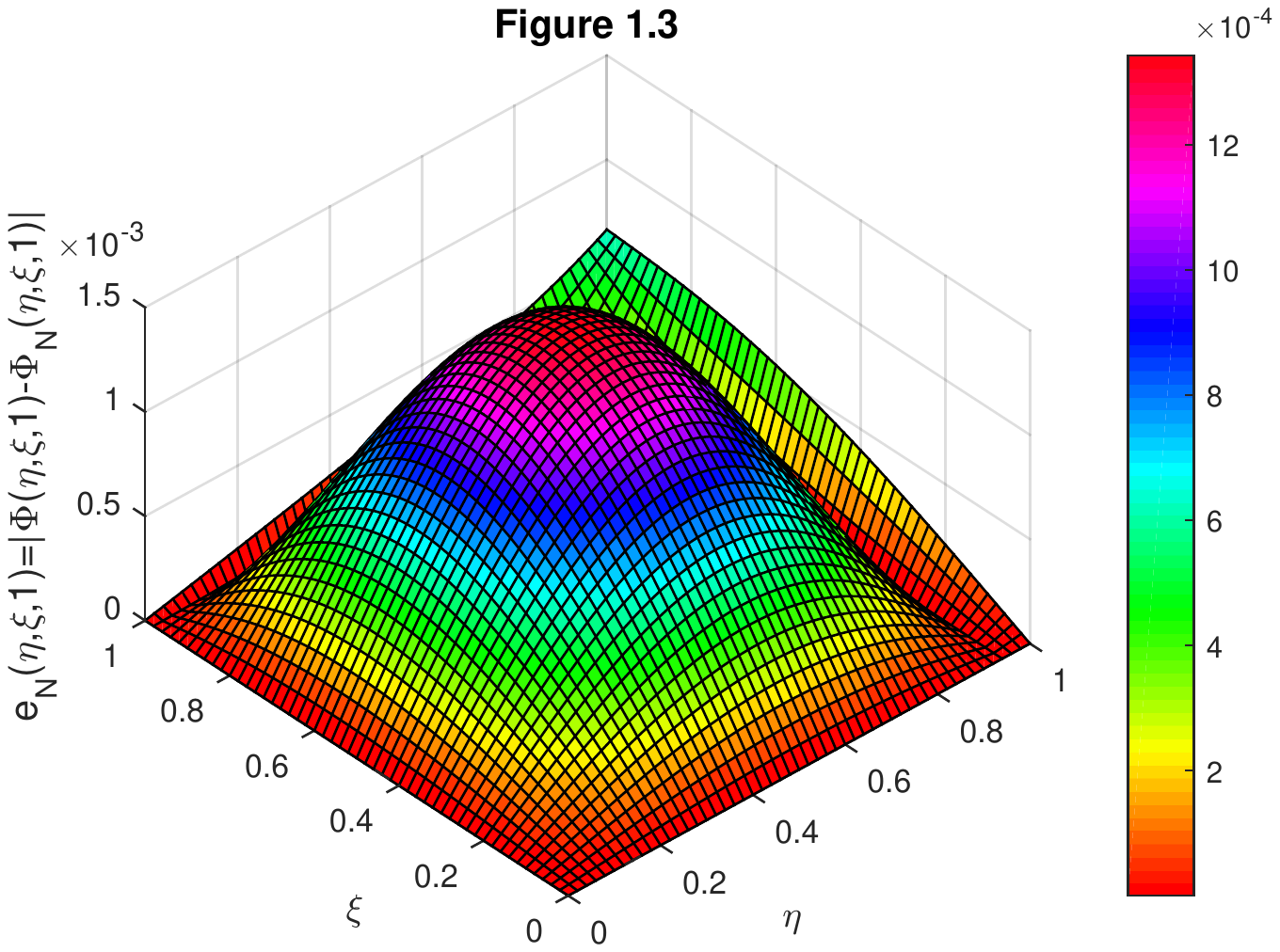}
	\vspace{-7cm}
	\caption{The spatio-temporal distribution of an error between the exact solution and its numerical solution utilising the LWOMM for Example 1 at $N=6$.} 
	\label{fig3}
\end{figure}
\begin{figure}[H]
	\vspace{-7cm}
	\includegraphics[width=\textwidth]{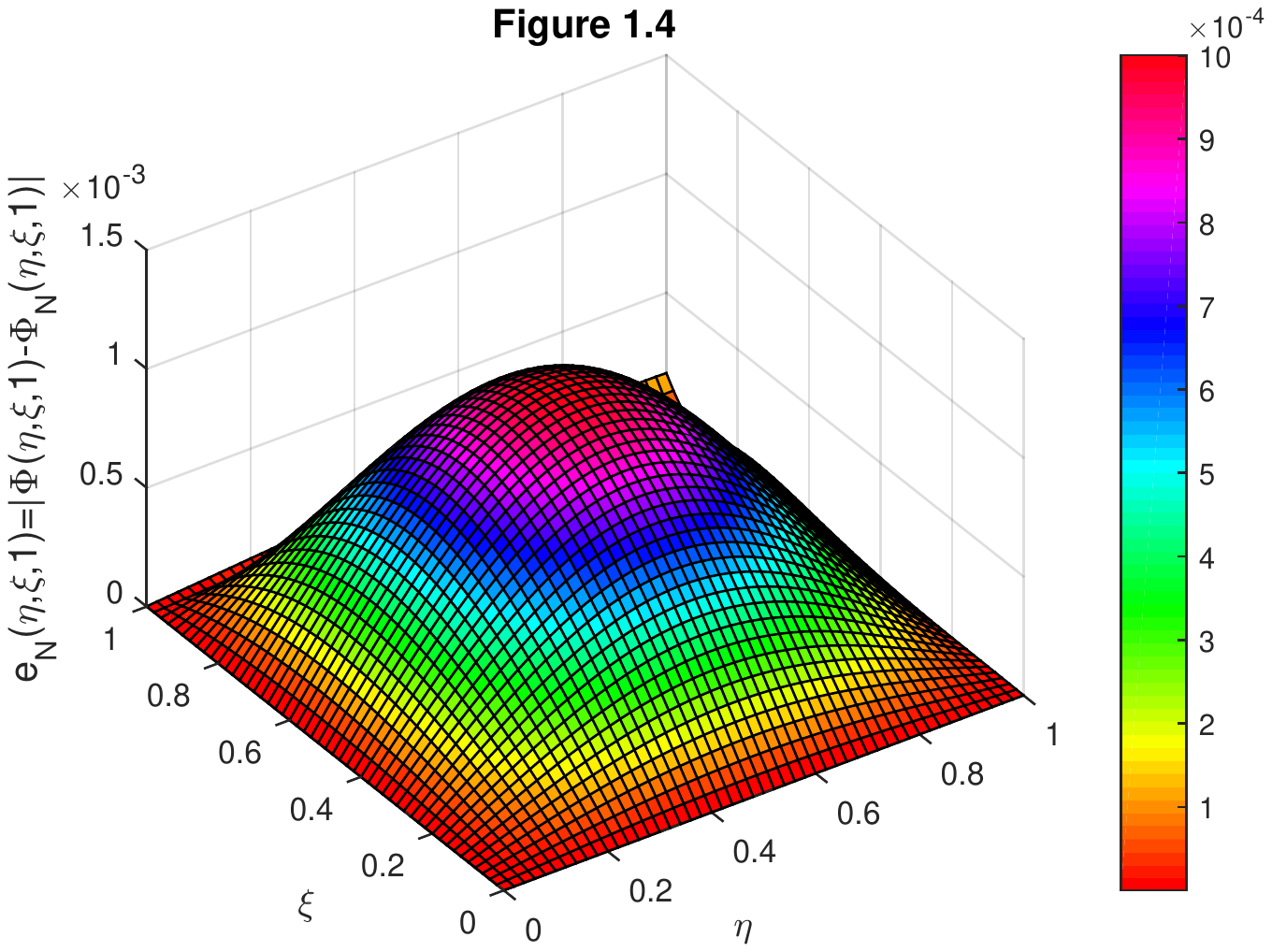}
	\vspace{-7cm}
	\caption{The spatio-temporal distribution of an error between the exact solution and its numerical solution utilising the LWOMM for Example 1 at $N=7$.} 
	\label{fig4}
\end{figure}
\begin{center}
\footnotesize
\captionof{table}{Absolute error of LWOMM at $e_{N}(\eta,\xi,1)$ for values of $N=4,5,6,7$ of Example 1 } 
\begin{tabular}{ ccccc }
\hline
$(\eta,\xi)$ & $N=4$ $N=5$ & $N=6$ & $N=7$ \\
\hline
(0.0,0.0) & 9.8020e-16 & 1.0604e-14 & 5.8261e-14  & 4.6324e-13\\
(0.1,0.1) & 7.5943e-05 & 2.4407e-04 & 1.0070e-04  & 7.5510e-05\\
(0.2,0.2) & 2.4400e-04 & 9.3702e-04 & 3.7520e-04  & 2.8222e-04\\
(0.3,0.3) & 3.1769e-04 & 2.0000e-03 & 7.4580e-04  & 5.5994e-04\\
(0.4,0.4) & 1.4792e-05 & 3.1000e-03 & 1.1000e-04  & 8.2115e-04\\
(0.5,0.5) & 1.0000e-03 & 4.2000e-03 & 1.3000e-03  & 9.7446e-04\\
(0.6,0.6) & 3.2000e-03 & 5.0000e-03 & 1.3000e-03  & 9.5378e-04\\
(0.7,0.7) & 6.7000e-02 & 5.2000e-03 & 9.6017e-04  & 7.4814e-04\\
(0.8,0.8) & 1.2100e-02 & 4.8000e-03 & 3.9005e-04  & 4.2101e-04\\
(0.9,0.9) & 1.9200e-02 & 3.8000e-03 & 2.5801e-04  & 9.4096e-05\\
(1.0,1.0) & 2.8100e-02 & 2.5000e-03 & 6.6554e-04  & 1.6198e-04\\
\hline  
\label{tab1}
\end{tabular}
\end{center}
\begin{center}
\footnotesize
	\captionof{table}{Norm error of LWOMM for values of $N=4,5,6,7$ of Example 1} \hspace{10cm}
\begin{tabular}{ cccccc }
\hline
Norm & $N=5$ & $N=4$ & $N=6$ & $N=7$ & Mittal and Bhatia \cite{mittal_2014}\\
\hline
$l^{2}$      & 2.3900e-02 & 1.1300e-02 & 2.6000e-03  & 2.1000e-03 & 1.7174e-04\\
$l^{\infty}$ & 1.9200e-02 & 5.2000e-03 & 1.3000e-03  & 9.7446e-04 & 5.6395e-04\\
\hline  
\label{tab2}
\end{tabular}
\end{center}
\newpage
\textbf{Example 2.} In this example, we consider the HTE Eq.(1) in the domain $(\eta,\xi,t)\in \Omega$, with $\lambda_{1} =\lambda_{2}=1$ and $F(\eta,\xi,t)=2\cos(t) \sin(\eta) \sin(\xi) - 2 \sin(t) \sin(\eta) \sin(\xi)$. The initial and boundary conditions are given by
\begin{equation}\label{54}
\left\{ \begin{array}{l} \Phi(\eta,\xi,0)=\sin(\eta)\sin(\xi), \vspace{0.2cm}\\
\Phi_{t}(\eta,\xi,0)=0 \vspace{0.2cm}
\end{array}\right.
\end{equation}
and
\begin{equation}\label{55}
\left\{ \begin{array}{l} \Phi(0,\xi,t)=0,~\Phi(1,\xi,t)=\cos(t)\sin(1)\sin(\xi) \vspace{0.2cm}\\
\Phi(\eta,0,t)=0,~\Phi(\eta,1,t)=\cos(t)\sin(\eta)\sin(1).
\vspace{0.2cm}
\end{array}\right.
\end{equation}
The exact solution is given by
\begin{equation*}
\Phi(\eta,\xi,t)=\cos{t}\sin(\eta)\sin(\xi).
\end{equation*}
\begin{figure}[H]
	\vspace{-7cm}
	\includegraphics[width=\textwidth]{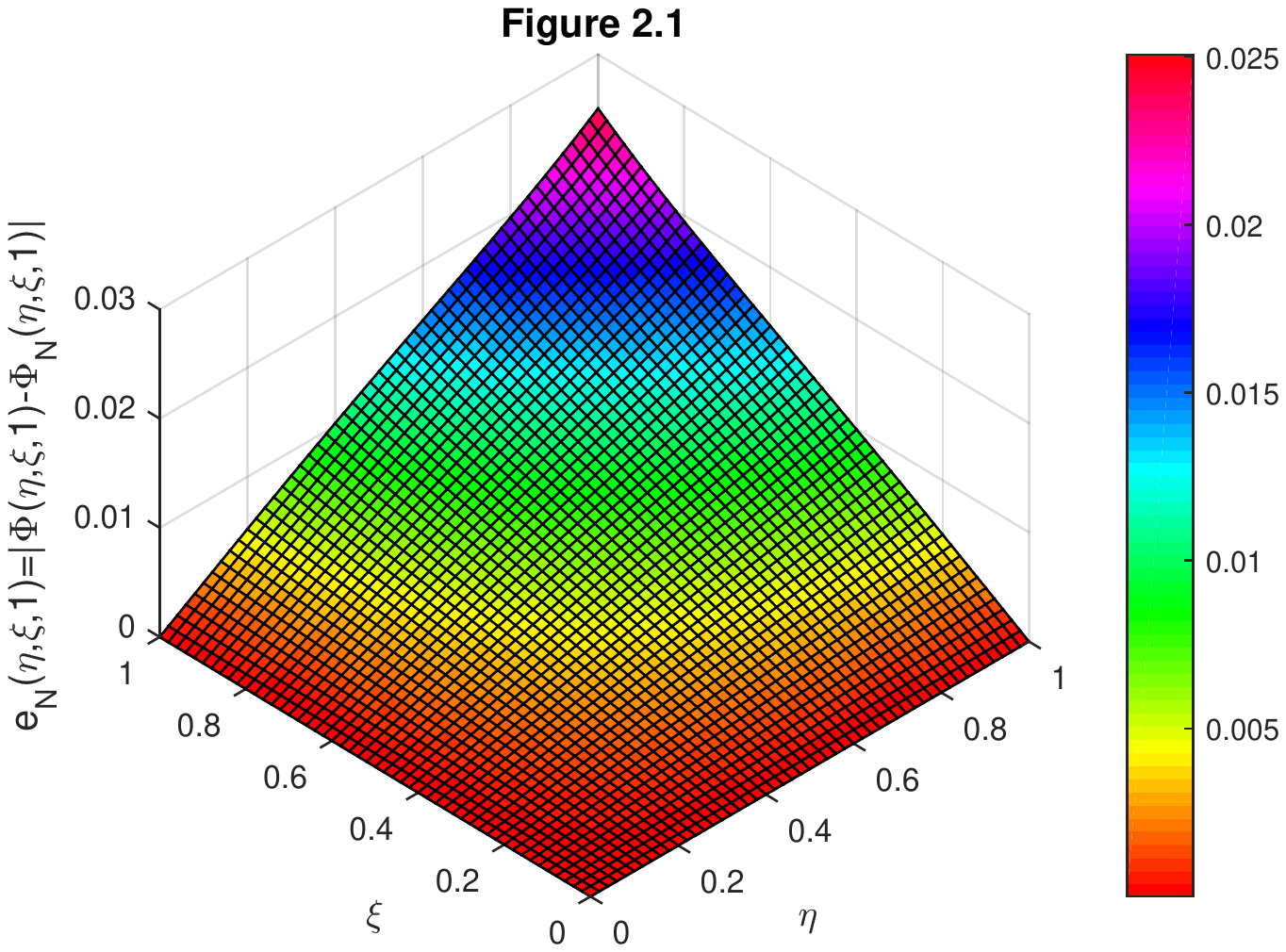}
	\vspace{-7cm}
	\caption{The spatio-temporal distribution of an error between the exact solution and its numerical solution utilising the LWOMM for Example 1 at $N=4$.} 
	\label{fig5}
\end{figure}

\begin{figure}[H]
	\vspace{-7cm}
	\includegraphics[width=\textwidth]{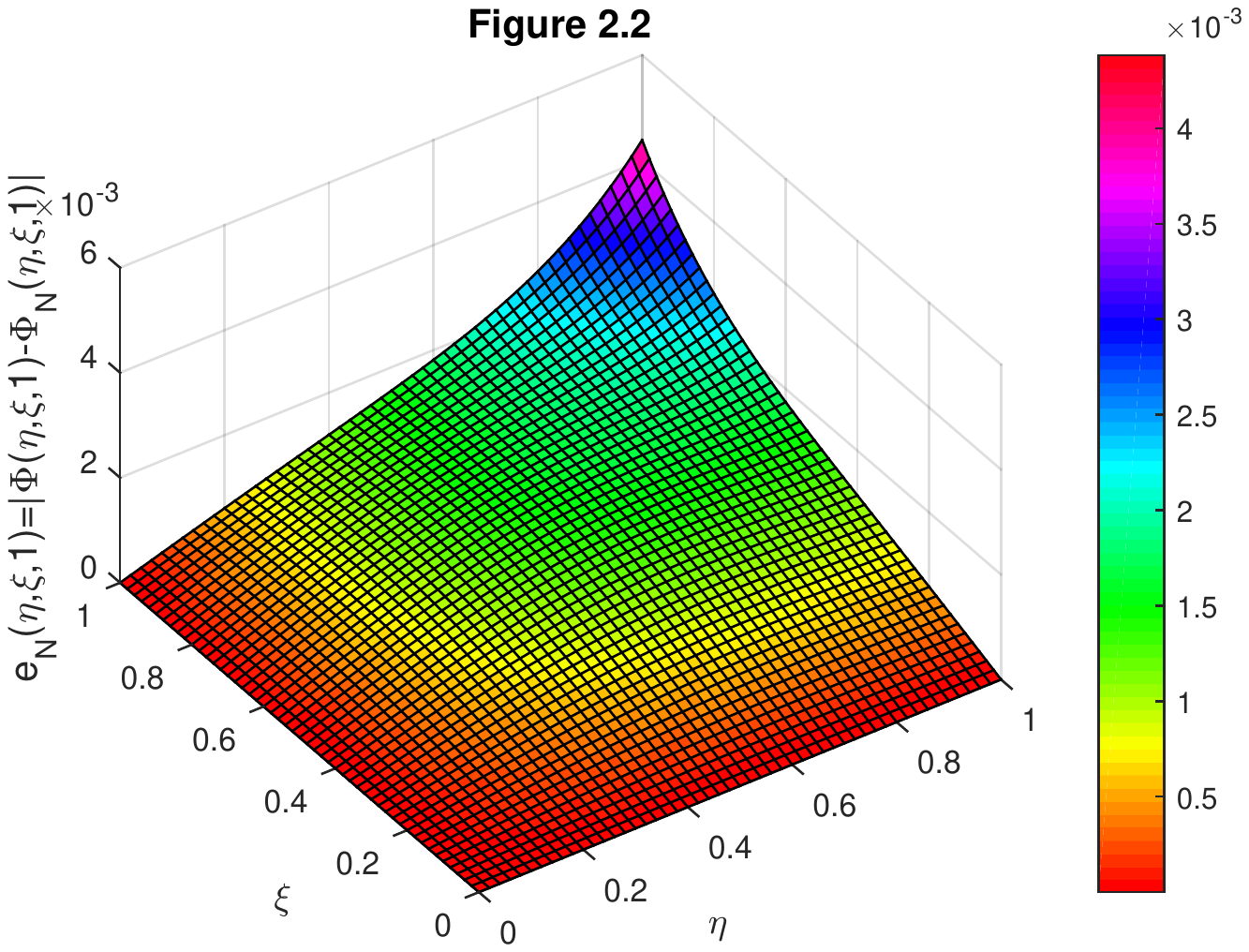}
	\vspace{-7cm}
	\caption{The spatio-temporal distribution of an error between the exact solution and its numerical solution utilising the LWOMM for Example 1 at $N=5$.} 
	\label{fig6}
\end{figure}

\begin{figure}[H]
	\vspace{-6cm}
	\includegraphics[width=\textwidth]{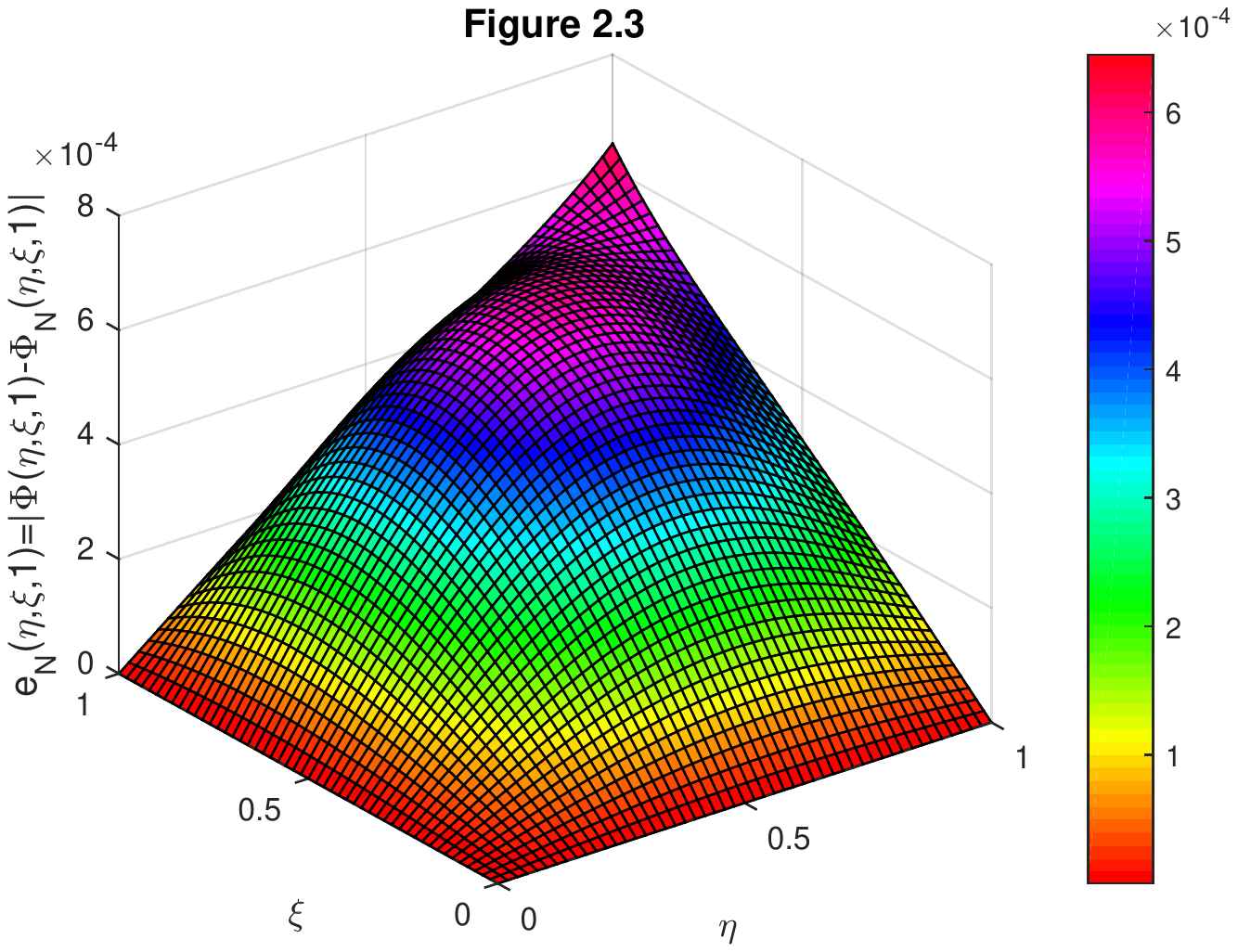}
	\vspace{-7cm}
	\caption{The spatio-temporal distribution of an error between the exact solution and its numerical solution utilising the LWOMM for Example 1 at $N=6$.} 
	\label{fig7}
\end{figure}
\begin{figure}[H]
	\vspace{-7cm}
	\includegraphics[width=\textwidth]{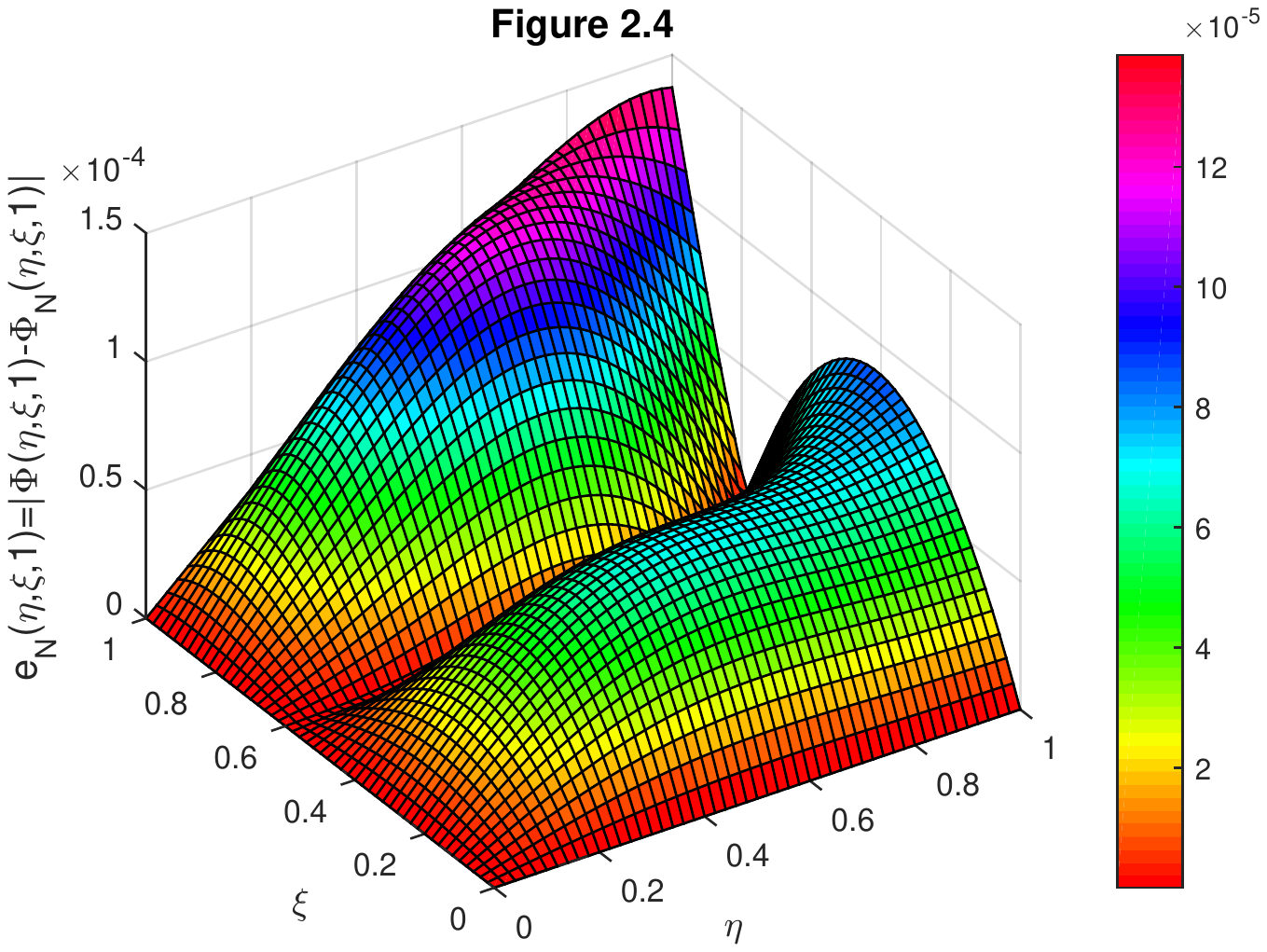}
	\vspace{-7cm}
	\caption{The spatio-temporal distribution of an error between the exact solution and its numerical solution utilising the LWOMM for Example 1 at $N=7$.} 
	\label{fig8}
\end{figure}
\begin{center}
\footnotesize
	\captionof{table}{Absolute error of LWOMM at $e_{N}(\eta,\xi,1)$ for values of $N=4,5,6,7$ of Example 2} 
\begin{tabular}{ ccccc }
\hline
$(\eta,\xi)$ & $N=4$ & $N=5$ & $N=6$ & $N=7$ \\
\hline
(0.0,0.0) & 6.6368e-14 & 1.0240e-15 & 3.9761e-18  & 3.8272e-18\\
(0.1,0.1) & 2.7356e-04 & 7.1939e-05 & 2.9464e-04  & 1.2611e-05\\
(0.2,0.2) & 1.1000e-03 & 2.8673e-04 & 1.1185e-04  & 3.7931e-05\\
(0.3,0.3) & 2.4000e-03 & 6.1143e-04 & 2.3039e-04  & 5.6171e-05\\
(0.4,0.4) & 4.1000e-03 & 9.8017e-04 & 3.6060e-04  & 5.3182e-05\\
(0.5,0.5) & 6.3000e-03 & 1.3000e-03 & 4.7503e-04  & 2.6537e-05\\
(0.6,0.6) & 8.8000e-03 & 1.6000e-03 & 5.5010e-04  & 1.3745e-05\\
(0.7,0.7) & 1.1800e-02 & 1.8000e-03 & 5.7469e-04  & 5.1243e-05\\
(0.8,0.8) & 1.5300e-02 & 2.0000e-03 & 5.6041e-04  & 7.5205e-05\\
(0.9,0.9) & 1.9600e-02 & 2.7000e-03 & 5.5311e-04  & 9.9891e-05\\
(1.0,1.0) & 2.5100e-02 & 4.4000e-03 & 6.4481e-04  & 1.3731e-05\\
\hline
\label{tab3}  
\end{tabular}
\end{center}
\begin{center}
\footnotesize
	\captionof{table}{Norm error of LWOMM for values of $N=4,5,6,7$ of Example 2} 
\begin{tabular}{ cccccc }
\hline
Norm & $N=4$ & $N=5$ & $N=6$ & $N=7$ & Mittal and Bhatia \cite{mittal_2014}\\
\hline
$l^{2}$      & 3.9100e-02 & 6.3000e-03 & 1.4000e-03  & 2.1077e-04 & 9.8870e-05\\
$l^{\infty}$ & 2.5100e-02 & 4.4000e-03 & 6.4481e-03  & 1.3731e-04 & 2.4964e-04\\
\hline
\label{tab4}  
\end{tabular}
\end{center}
\newpage
\textbf{Example 3.} In this example, we consider the HTE Eq.(1) in the domain $(\eta,\xi,t)\in \Omega$ with $\lambda_{1}=10, \lambda_{2}=5$ and $F(\eta,\xi,t)=22\cos(t)\sinh(\eta)\sinh(\xi) - 20\sin(t)\sinh(\eta)\sinh(\xi)$. The initial and
Dirichlet boundary conditions are given below
\begin{equation}\label{56}
\left\{ \begin{array}{l} \Phi(\eta,\xi,0)=\sinh(\eta)\sinh(\xi), \vspace{0.2cm}\\
\Phi_{t}(\eta,\xi,0)=0 \vspace{0.2cm}
\end{array}\right.
\end{equation}
and
\begin{equation}\label{57}
\left\{ \begin{array}{l} \Phi(0,\xi,t)=0,~\Phi(1,\xi,t)=\cos(t)\sinh(1)\sinh(\xi) \vspace{0.2cm}\\
\Phi(\eta,0,t)=0,~\Phi(\eta,1,t)=\cos(t)\sinh(\eta)\sinh(1)
\vspace{0.2cm}
\end{array}\right.
\end{equation}
The exact solution is given by
\begin{equation*}
\Phi(\eta,\xi,t)=\cos{t}\sinh(\eta)\sinh(\xi).
\end{equation*}
\begin{figure}[H]
	\vspace{-7cm}
	\includegraphics[width=\textwidth]{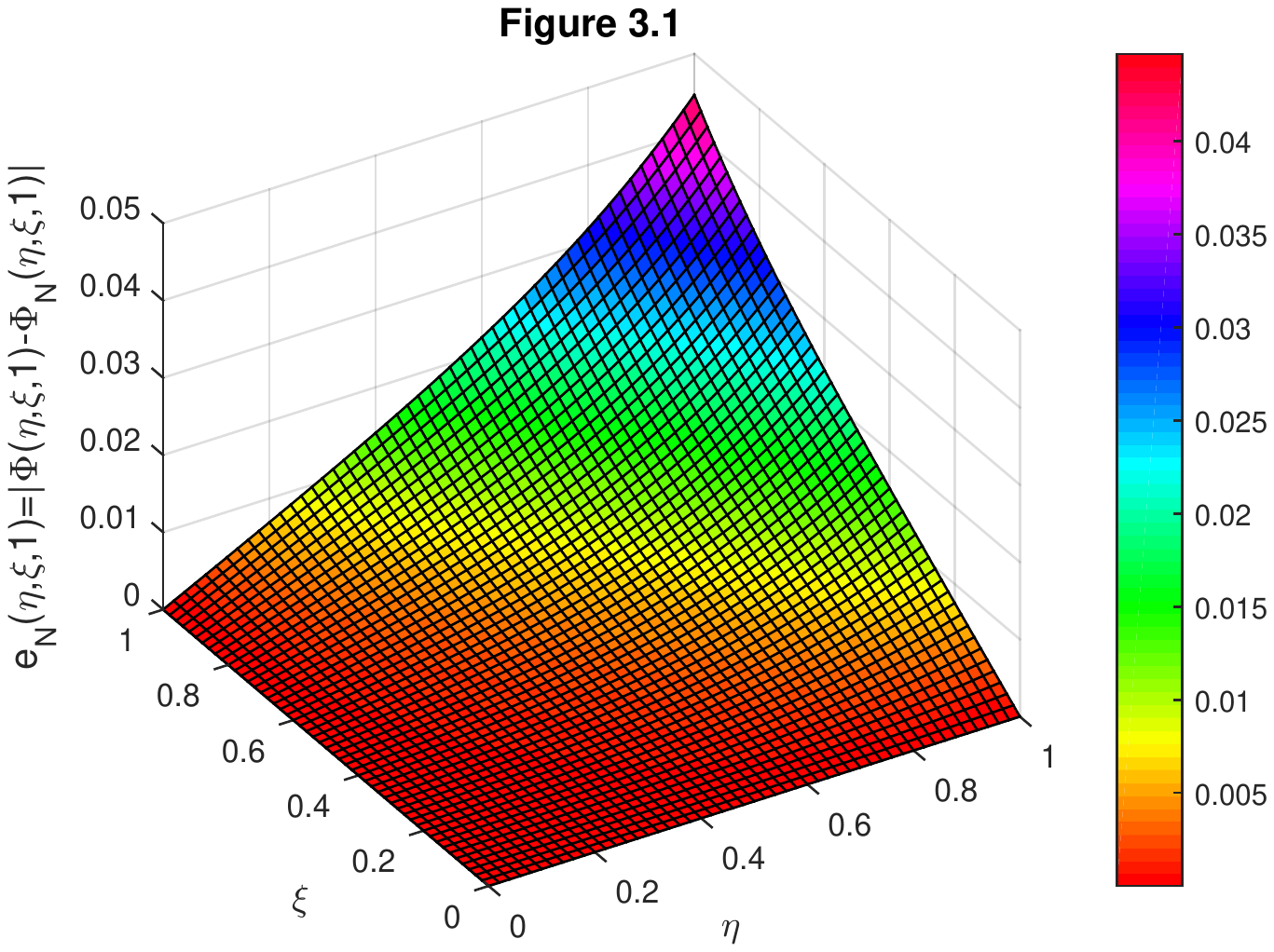}
	\vspace{-7cm}
	\caption{The spatio-temporal distribution of an error between the exact solution and its numerical solution utilising the LWOMM for Example 1 at $N=4$.} 
	\label{fig9}
\end{figure}

\begin{figure}[H]
	\vspace{-7cm}
	\includegraphics[width=\textwidth]{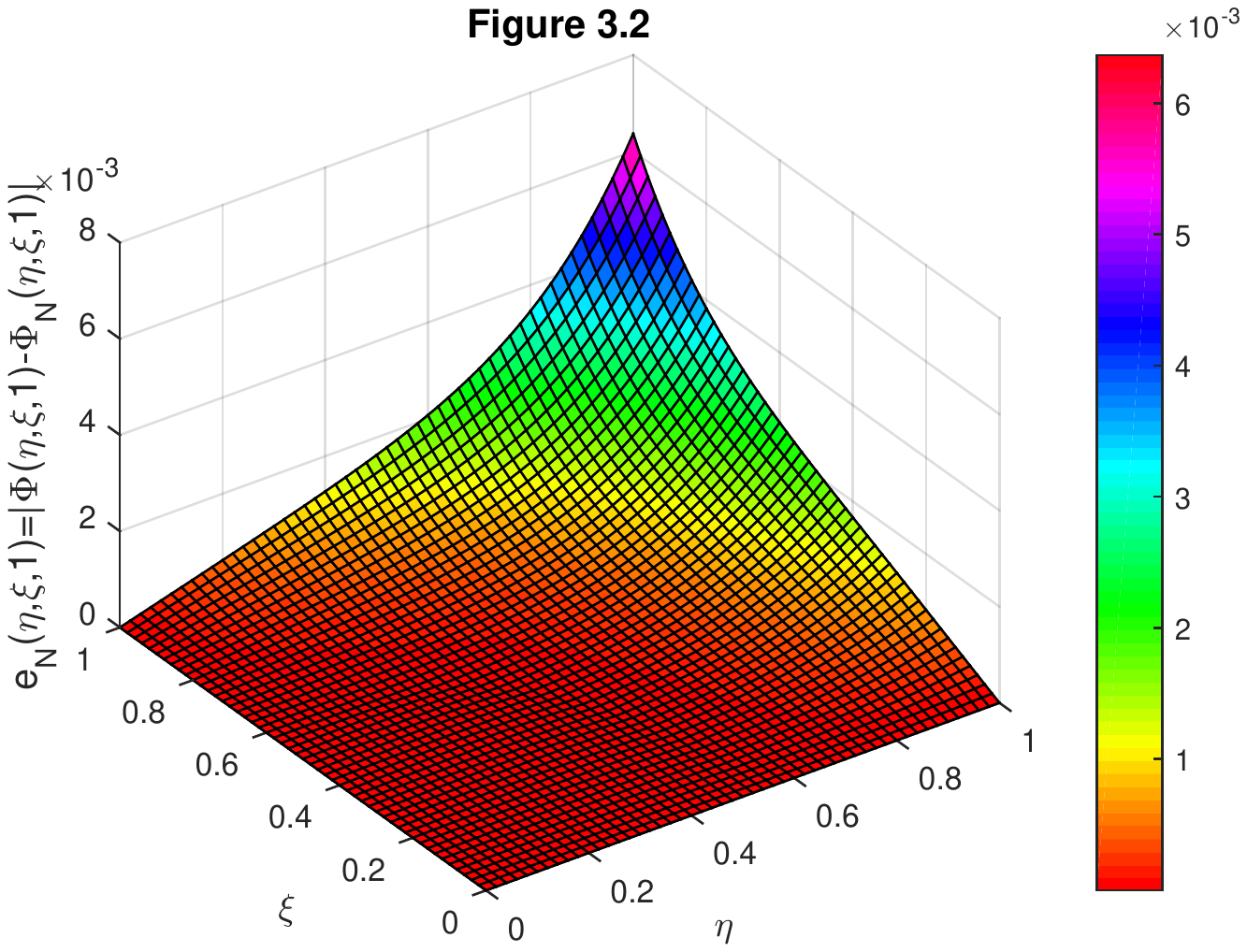}
	\vspace{-7cm}
	\caption{The spatio-temporal distribution of an error between the exact solution and its numerical solution utilising the LWOMM for Example 1 at $N=5$.} 
	\label{fig10}
\end{figure}

\begin{figure}[H]
	\vspace{-7cm}
	\includegraphics[width=\textwidth]{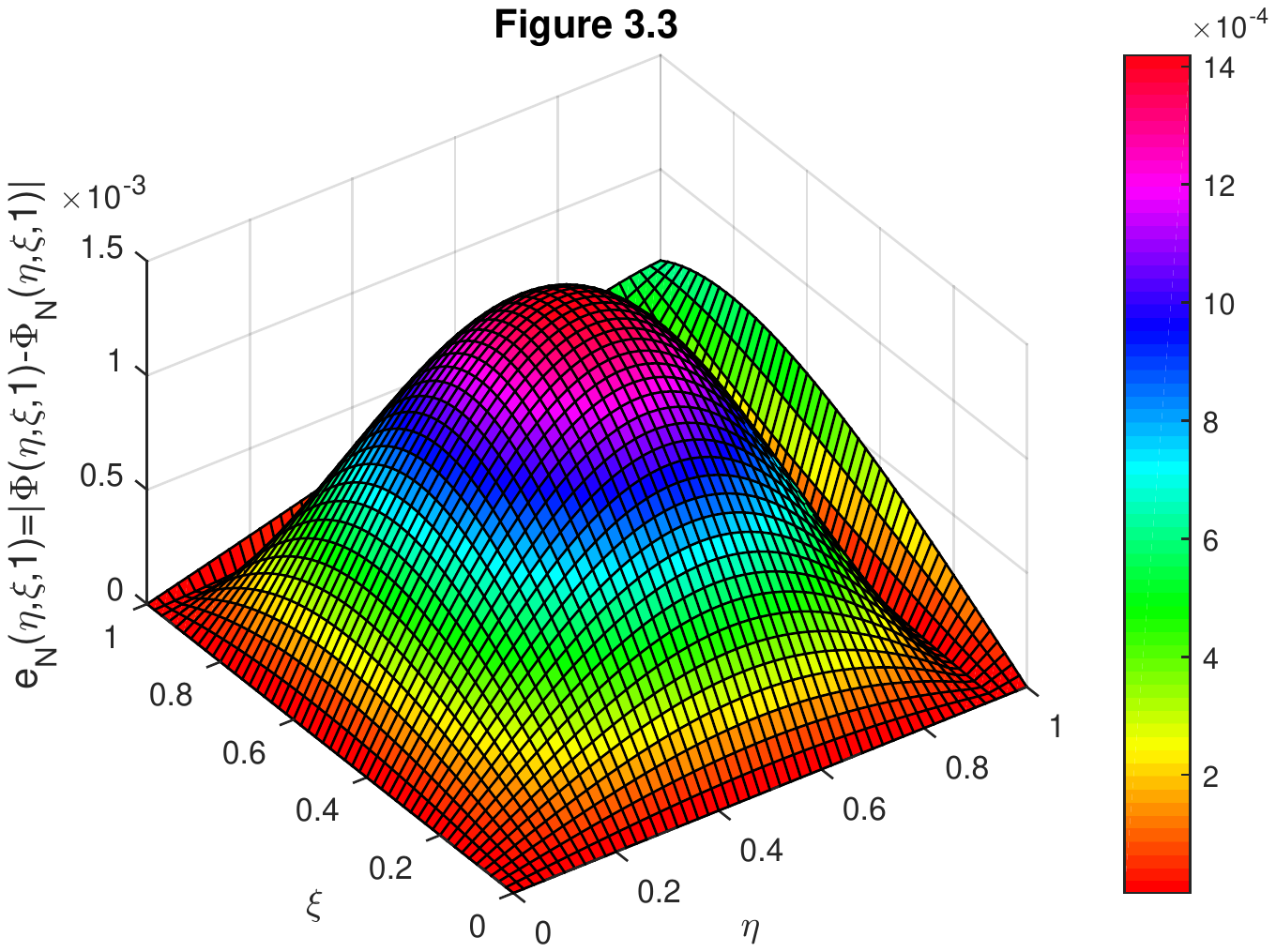}
	\vspace{-7cm}
	\caption{The spatio-temporal distribution of an error between the exact solution and its numerical solution utilising the LWOMM for Example 1 at $N=6$.} 
	\label{fig11}
\end{figure}
\begin{figure}[H]
	\vspace{-6cm}
	\includegraphics[width=\textwidth]{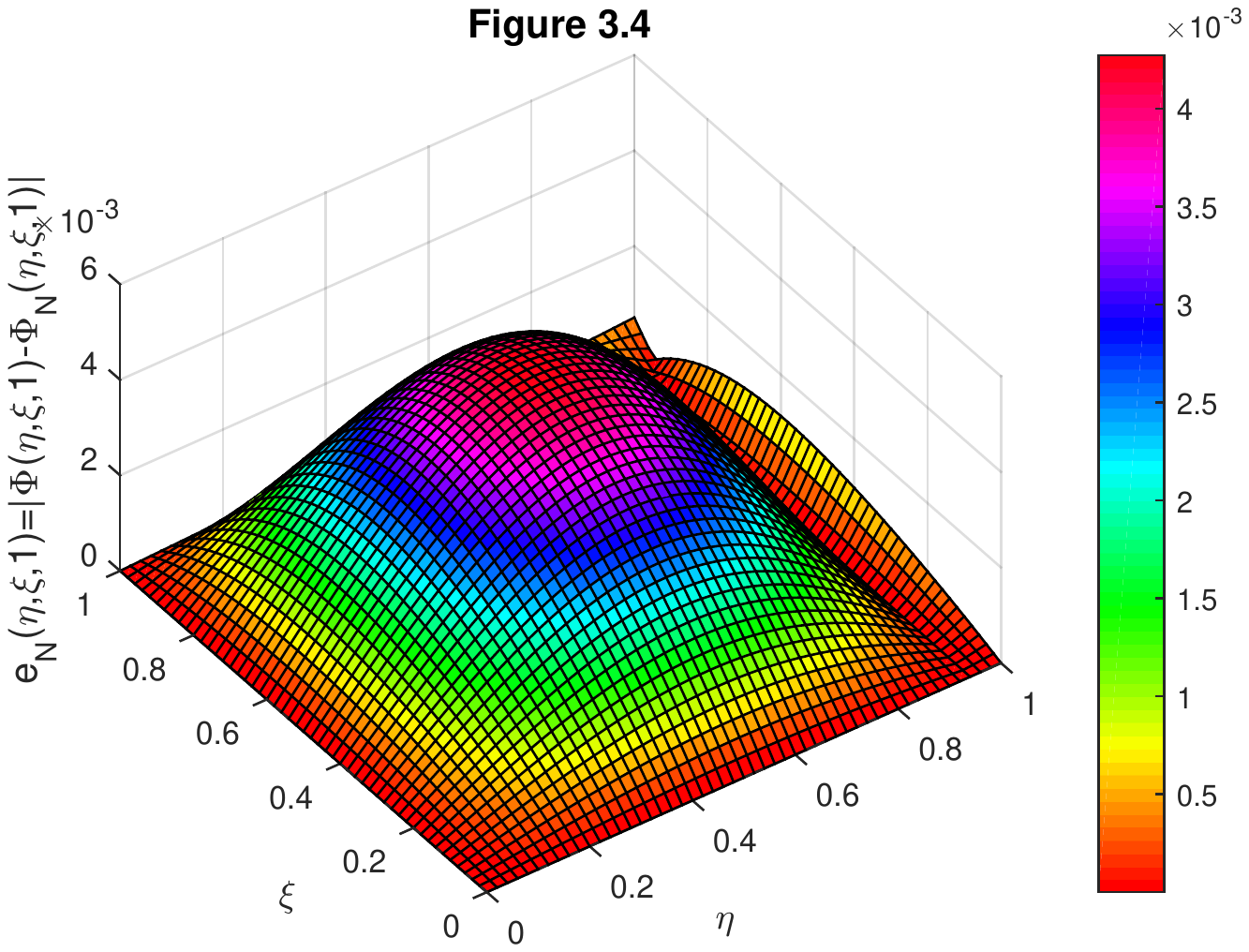}
	\vspace{-7cm}
	\caption{The spatio-temporal distribution of an error between the exact solution and its numerical solution utilising the LWOMM for Example 1 at $N=7$.} 
	\label{fig12}
\end{figure}
\begin{center}
\footnotesize
\captionof{table}{Absolute error of LWOMM at $e_{N}(\eta,\xi,1)$ for values of $N=4,5,6,7$ of Example 3} 
\begin{tabular}{ ccccc }
\hline
$(\eta,\xi)$ & $N=4$ & $N=5$ & $N=6$ & $N=7$ \\
\hline
(0.0,0.0) & 6.7126e-14 & 1.0492e-15 & 8.3206e-17  & 1.0593e-18\\
(0.1,0.1) & 6.3835e-05 & 9.9583e-06 & 3.3025e-04  & 1.0929e-04\\
(0.2,0.2) & 3.1414e-04 & 3.1003e-05 & 1.2000e-03  & 4.0642e-04\\
(0.3,0.3) & 9.3733e-04 & 4.3952e-05 & 1.3000e-03  & 8.0493e-04\\
(0.4,0.4) & 2.2000e-03 & 2.2048e-05 & 3.4000e-03  & 1.2000e-03\\
(0.5,0.5) & 4.5000e-03 & 7.3449e-05 & 4.1000e-03  & 1.4000e-03\\
(0.6,0.6) & 8.1000e-03 & 3.0522e-04 & 4.1000e-03  & 1.3000e-03\\
(0.7,0.7) & 1.3600e-02 & 7.8217e-04 & 3.4000e-03  & 9.6647e-04\\
(0.8,0.8) & 2.1300e-02 & 1.7000e-03 & 2.1000e-03  & 3.4296e-04\\
(0.9,0.9) & 3.1600e-02 & 3.4000e-03 & 6.7019e-04  & 3.1452e-04\\
(1.0,1.0) & 4.4700e-02 & 6.4000e-03 & 5.0796e-04  & 6.0130e-04\\
\hline
\label{tab5}  
\end{tabular}
\end{center}
\begin{center}
\footnotesize
\captionof{table}{Norm error of LWOMM for values of $N=4,5,6,7$ of Example 3} 
\begin{tabular}{ cccccc }
\hline
Norm & $N=4$ & $N=5$ & $N=6$ & $N=7$ & Mittal and Bhatia \cite{mittal_2014}\\
\hline
$l^{2}$      & 6.0900e-02 & 7.5000e-03 & 1.0500e-03  & 2.7000e-03 & 1.6144e-03\\
$l^{\infty}$ & 4.4700e-02 & 6.4000e-03 & 6.7000e-04  & 1.4000e-04 & 3.6000e-04\\
\hline 
\label{tab6} 
\end{tabular}
\end{center}

\begin{center}
\footnotesize
\captionof{table}{Norm error of LWOMM for values of $N=4,5,6,7$ of Example 3} 
\begin{tabular}{ cccccc }
\hline
 \textbf{n} & \textbf{$x_{i}$}  & \textbf{$f(x_{i})$}  & \textbf{$f'(x_{i})$} & \textbf{$x_{i+1}$} & \textbf{$|x_{i+1}-x_{i}|$}  \\ 
\hline
 0 & $x_{0}=6.0$   & \textbf{$f(x_{0})$}=32.0 & \textbf{$f'(x_{0})$}=12.0 & \textbf{$x_{1}$}=3.33 & - \\ 
 1 & $x_{1}=3.33$  & \textbf{$f(x_{1})$}=7.09 & \textbf{$f'(x_{1})$}=6.66 & \textbf{$x_{2}$}=2.27 & 1.06 \\ 
 2 & $x_{2}=2.27$  & \textbf{$f(x_{2})$}=1.15 & \textbf{$f'(x_{2})$}=4.54 & \textbf{$x_{3}$}=2.01 & 0.26 \\ 
 3 & $x_{3}=2.01$  & \textbf{$f(x_{3})$}=0.04 & \textbf{$f'(x_{3})$}=4.02 & \textbf{$x_{4}$}=2.00 & 0.01 \\ 
\hline
\label{tab7}  
\end{tabular}
\end{center}
\section{Conclusion}
The telegraph equations are of special interest as they are used to understand various physical and complex phenomena. Two-dimensional hyperbolic telegraph equations (HTE) are usually difficult to solve analytically. In this article, Legendre wavelet matrix method (LWOMM) has been developed to tackle the two-dimensional HTE with Dirichlet boundary conditions. For this purpose, we constructed operational matrices based on Legendre wavelet for integration and differentiation for the solution of HTE. After implementing of LWOMM  on HTE, HTE converted into algebraic generalized Sylvester equation which is solved by  BICGSTAB method. The results of the numerical solution of Eq.(\ref{1})   for a small number of Legendre wavelet basis functions are illustrated in the form of $Tables ~\ref{tab1}-\ref{tab6}$ and the $Figures~ \ref{fig1}-\ref{fig12}$.  Also, we have comprised propose method  pointwise errors with the some existing method \cite{mittal_2014} for different values of  $\lambda_{1}, \lambda_{2}, k, M, k', M', k''$, and $M''$.  Numerical results of absolute errors, $l^{2}~ \&~ l^{\infty}$ errors are provided in $Tables~ \ref{1},\ref{tab3} ~\&~\ref{tab5} $ and $Table~\ref{tab2},\ref{tab4}~\&~\ref{tab6}$ respectively . Moreover, theoretically convergence analysis for the solution approximation is provided. So, we can say that the proposed method is very easy, gives better accuracy at less time and efficient at a small number of Legendre wavelet basis by using convergence analysis, absolute error, $l^{2}$ error,  $l^{\infty}$ error, error tables and error graphs.

\section*{Acknowledgement}
The first author acknowledge the financial support under National Postdoctoral Fellowship from Science and Engineering Research Board, India, with sanction file no. PDF/2019/001275.

\section*{References}

\end{document}